\documentclass{amsart}

\usepackage{amsmath}
\usepackage{amsthm} 
\usepackage{graphicx}
\usepackage{color}
\usepackage{scalerel}
\usepackage[dvipsnames]{xcolor}
\usepackage{stmaryrd}
\usepackage{relsize}
\usepackage{comment}

\usepackage{enumerate}

\usepackage{amsfonts}
\usepackage{amssymb}
\usepackage[hidelinks]{hyperref}
\usepackage[capitalize]{cleveref}

\usepackage{nicefrac}
\usepackage{xfrac}

\usepackage{tipa}
\usepackage[normalem]{ulem}

\usepackage{mathtools}

\DeclareMathOperator*{\forkindep}{\raise0.2ex\hbox{\ooalign{\hidewidth$\vert$\hidewidth\cr\raise-0.9ex\hbox{$\smile$}}}}

\newcommand{\tp}{\operatorname{tp}}

\newcommand{\Av}{\mathrm{Av}}
\newcommand{\Th}{\mathrm{Th}}

\newcommand{\Lc}{\mathcal{L}}

\newcommand{\weak}{adequate}
\newcommand{\strong}{excellent}

\newcommand{\io}{\mathrm{i.o.}}

\newtheorem*{claim-star}{Claim}
\newtheorem*{theorem-non}{Theorem}
\newtheorem{theorem}{Theorem}[section] 
\newtheorem{lemma}[theorem]{Lemma}

\newtheorem{prop-def}[theorem]{Proposition-Definition}
\newtheorem{corollary}[theorem]{Corollary}
\newtheorem{fact}[theorem]{Fact}
\newtheorem{fact-eh}[theorem]{Fact(?)}

\newtheorem{question}[theorem]{Question}
\newtheorem{proposition}[theorem]{Proposition}
\newtheorem{proposition-eh}[theorem]{Proposition(?)}
\newtheorem*{theorem-star}{Theorem}
\newtheorem*{conjecture-star}{Conjecture}
\newtheorem*{lemma-star}{Lemma}

\theoremstyle{definition}
\newtheorem{definition}[theorem]{Definition}
\newtheorem{example}[theorem]{Example}

\newtheorem{remark}[theorem]{Remark}
\theoremstyle{remark}

\newtheorem*{warning}{Warning}

\newcommand{\inv}{\mathrm{inv}}
\newcommand{\supp}{\mathrm{supp}}

\newcommand{\st}{\mathrm{st}}
\newcommand{\diam}{\mathrm{diag}}
\newcommand{\qf}{\mathrm{qf}}
\newcommand{\str}{\mathrm{Str}}
\newcommand{\fin}{\mathrm{fin}}

\DeclareMathOperator{\Sym}{Sym}

\allowdisplaybreaks 

\title{Model theoretic events }
\author[K. Gannon]{Kyle Gannon$^{\dagger}$}
\thanks{$^{\dagger}$ Supported by the Fundamental Research Funds for the Central Universities, Peking University, grant no. 7100604835 and by the National Natural Science Fund of China, grant no. 12501001.}
\address{$^{\dagger}$ Beijing International Center for Mathematical Research (BICMR) \\ Peking University \\ Beijing, China.}
\email{kgannon@bicmr.pku.edu.cn} 
\author[J. Hanson]{James E. Hanson}
\address{$^{\ast}$  
  Iowa State University \\
  Ames, IA, USA }
\email{jameseh@iastate.edu} 

\begin{document}

\begin{abstract}


We develop a notion of sampling, called \emph{generic sampling}, for the context of global Keisler measures where the standard product is replaced by the Morley product. Choosing a point randomly in this space with respect to our distribution yields a \emph{random generic type} in infinitely many variables. We investigate several natural model-theoretic events and provide conditions under which they occur for almost all random generic types.

\end{abstract}

\maketitle

\section{Introduction}

Over the last 15 years, the study of Keisler measures has become an increasingly active and applicable topic in model theory. In the mid-1980s, Keisler originally studied measures in the NIP context as a way to generalize the theory of forking \cite{keisler1987measures}. Around 20 years later, these objects were revisited and recontextualized in a series of papers with authors among Hrushovski, Pillay, Peterzil, and Simon \cite{NIP1,NIP2,NIP3}. In the current ethos, Keisler measures are a key tool for studying definable groups (see e.g., \cite{chernikov2022definable,CS,conant2020group}), for applying model theory to combinatorics (e.g., regularity lemmas; see \cite{NIP5,NIP4,malliaris2016stable}), and for pure model theory (see e.g., \cite{ben2009randomizations,chernikov2023invariant,CGH,hrushovski2020first}).

We introduce the concept of \emph{generic sampling}, which builds on ordinary independent sampling relative to a probability distribution: given a probability space $(X,\mathbf{B},\mu)$, an \emph{$n$-sample} is an $n$-tuple from $X$ drawn according to the product measure $\mu^{n}$, and an \emph{infinite sample} is an element of the infinite product space $X^{\mathbb{N}}$ chosen with respect to $\mu^{\mathbb{N}}$ — or, equivalently, as a sequence obtained by independently drawing points from $X$ with distribution $\mu$. In our setting, $X = S_{x}(\mathcal{U})$ is the space of global types in a single variable, endowed with its Borel $\sigma$-algebra. However, the sample space $S_{x}(\mathcal{U})^{\mathbb{N}}$ does not contain enough \emph{model theoretic data}. To remedy this, we combine our distribution $\mu$ with the Morley product to construct a probability measure $\mathbb{P}_{\mu}$ on the space $S_{\mathbf{x}}(\mathcal{U})$ of types in countably many variables $\mathbf{x} = (x_i)_{i \geq 1}$. This probability space is rich enough to measure many important model theoretic events. We remark that these kinds of measures have been studied before \cite{NIP3,conant2025generic}, but not explicitly through the lens of probability theory.

Furthermore, this perspective yields a concrete connection between Keisler measures and measures which naturally arise in areas of combinatorics and probability theory. More explicitly, we demonstrate how to use Keisler measures and the Morley product to construct measures on the space of labeled $\mathcal{L}$-structures (Remark \ref{remark:connection}). Such measures are central to the study of graph limits (graphons, hypergraphons), random structures, and the theory of exchangeable random arrays (Aldous–Hoover). These connections are begin pursued in ongoing projects.

A \emph{model theoretic event} is a $\mathbb{P}_{\mu}$-measurable subset of $S_{\mathbf{x}}(\mathcal{U})$. We use the term \emph{random generic type} to refer to an element of the support of $\mathbb{P}_{\mu}$. In this article, we are concerned with two distinct kinds of events: 
\begin{enumerate}
    \item For a fixed $\mathcal{L}$-structure $N$, what is the probability that the induced structure on a random generic type is isomorphic to $N$?
    \item What is the probability that a random generic type witnesses a fixed dividing line (e.g., instability, IP, etc.)?
\end{enumerate}


Towards resolving the first question, we identify a measure-theoretic version of \emph{one-point extension property} which we call the \emph{excellent measure extension axiom}. We prove that in a finite relational language, if $\mu$ is 
a global measure which satisfies the excellent measure extension axiom, then 
\begin{enumerate}
    \item There exists a countable structure $N$ such that the induced structure on $\mathbb{P}_{\mu}$-almost all random generic types is isomorphic to $N$. 
    \item $\Th_{\mathcal{L}}(N)$ is countably categorical. 
    \item The associated measure on the space of labeled $\mathcal{L}$-structures is \emph{exchangeable}, i.e., invariant under the natural action of $\Sym(\mathbb{N})$.
\end{enumerate}


The second type of event of interest are \emph{witnesses to dividing lines}. While there are several plausible interpretations of this notion, we focus on a strong variant among the possible definitions. For example, given a formula $\varphi(x,y)$, we say that a type $p \in S_{\mathbf{x}}(\mathcal{U})$ \emph{strongly witnesses the instability of $\varphi$} if for any/all realizations $(a_i)_{i < \omega} \models p$, there exists $(b_j)_{j < \omega}$ such that 
\begin{equation*}
    \models \varphi(a_i,b_j) \Longleftrightarrow i \leq j. 
\end{equation*}
Notice that we can define \emph{strongly witnessing $k$-instability of $\varphi$} as follows:
\begin{equation*}
    O_{k}^{\varphi}(x_1,\ldots,x_k) \coloneqq \exists y_1,\ldots,y_k\left( \bigwedge_{1 \leq i \leq j \leq k}  \varphi(x_i,y_j) \wedge \bigwedge_{1 \leq j < i \leq k} \neg \varphi(x_i,y_j) \right).
\end{equation*}
Then a type $p \in S_{\mathbf{x}}(\mathcal{U})$ strongly witnesses the instability of $\varphi$ if and only if $p \in \mathbf{O}^{\varphi}$ where $\mathbf{O}^{\varphi} \coloneqq \bigcap_{k = 1}^{\infty} [O_{k}^{\varphi}(\bar{x})]$. The probability of a random generic type strongly witnessing instability is the measure of $\mathbf{O} \coloneqq \bigcup_{\varphi \in \mathcal{L}} \mathbf{O}^{\varphi}$. 
Among other things, we observe that $\mathbb{P}_{\mu}(\mathbf{O}) \in \{0,1\}$ and prove that if $\mu$ is \emph{generically stable} (see Definition \ref{cheat}(7)), then $\mathbb{P}_{\mu}(\mathbf{O}) = 0$ (Theorem \ref{theorem:witness}). 


Moreover, we prove that if $\varphi(x,y)$ is NIP, then the average value across all permutations of the measure of $O_{n}^{\varphi}(x_{\sigma(1)},\ldots,x_{\sigma(n)})$ must tend toward 0. More explicitly, we prove that if $\mu \in \mathfrak{M}_{x}^{\inv}(\mathcal{U},M)$ is $M$-adequate (i.e. Borel-definable and self-associative) and $\varphi(x,y)$ is NIP, then 
\begin{equation*}
    \lim_{n \to \infty} \frac{1}{n!} \sum_{\sigma \in \Sym(n)} \mathbb{P}_{\mu}(O_{n}^{\varphi}(x_{\sigma(1)},\ldots,x_{\sigma(n)})) = 0. 
\end{equation*}
We remark that when $\varphi(x,y)$ is not NIP, then it is possible for the limit to converge to $1$. This occurs for the \emph{weighted-coin-flipping} measures on the Rado graph. 


The paper is outlined as follows: In Section 2, we discuss preliminaries. In Section 3, we introduce the notion of generic sampling. Section 3 also contains some important computational lemmas which will be relevant in later sections of the paper.
Section 4 focuses on the \emph{isomorphism problem}, i.e. given a Keisler measure $\mu$, when does there exist some $\mathcal{L}$-structure $N$ such that the induced structure on almost all random generic types is isomorphic to $N$? We restrict ourselves to the case where our language is relational. We first prove that for any fixed $\mathcal{L}$-structure $N$ and any \emph{adequate measure} $\mu$, the set ${p \in S_{\mathbf{x}}(\mathcal{U}): \text{induced structure on } p \text{ is isomorphic to } N}$ is $\mathbb{P}_{\mu}$-measurable (Proposition \ref{prop:borel}). Next, we describe some relatively easy contexts in which one can obtain a positive answer to the isomorphism problem. We then move on to the main portion of the section where we introduce the \emph{adequate} and \emph{excellent measure extension axioms}. We prove that if a measure $\mu$ witnesses the excellent measure extension axiom, then we obtain a positive answer to the isomorphism problem and moreover the theory of the model witnessing this property is $\aleph_0$-categorical (Theorem \ref{theorem:main}, Theorem \ref{theorem:cat}). Furthermore, the associated measure on the space of labeled $\mathcal{L}$-structures is $\Sym(\mathbb{N})$-invariant. Section 5 is focused on concrete examples around the isomorphism problem. We show directly that the \emph{model-theoretic} Lebesgue measure over DLO and the \emph{weighted-coin-flipping} measures on the Rado graph have positive solutions to the isomorphism problem (Example \ref{example:random}, Example \ref{example:DLO}). We also exhibit some examples of measures which admit a negative answer to the isomorphism problem (Example \ref{example:bad}, Example \ref{example:bad2}). In the final section, we focus our attention on another kind of event, those associated with witnessing dividing lines. We prove that if $\mu$ is \emph{generically stable}, then almost no random generic types witness instability, the independence property, or the strict order property (Theorem \ref{theorem:witness}). This follows in part from the observation that witnessing any of these dividing lines for a particular formula has probability $0$ or $1$ (Lemma \ref{lem:01}). We then consider some concrete examples. Finally, we focus on the NIP setting. We prove that while it is possible that almost all random generic types witness instability in this setting, the average over a certain family of permutations converges to 0 (Theorem \ref{theorem:average}).

\section*{Acknowledgements}
We would like to thank Nathanael Ackerman, Cameron Freer, and Rehana Patel for helpful discussions, comments, and suggestions, specifically for discussions involving Theorem \ref{theorem:AFP}.

\section{Preliminaries} 

Given real numbers $r,s$ and a real number $\epsilon > 0$, we write $r \approx_{\epsilon} s$ to mean $|r - s| < \epsilon$.

We assume some familiarity with model theory. A good reference for background is \cite{Guide}. We will always have a fixed language $\mathcal{L}$ and a fixed $\mathcal{L}$-theory $T$ in the background. The symbol $\mathcal{U}$ will denote a monster model of $T$ and $M$ will denote a small elementary submodel. The symbols $x,y,z\ldots$ will denote singleton variables, $\bar{x}, \bar{y},\bar{z}\ldots$ will denote tuples of variables, and $a,b,c\ldots$ denote parameters. If $A \subseteq \mathcal{U}$, then $\mathcal{L}_{x}(A)$ is the Boolean algebra of formulas with free variables among $x$ and parameters from $A$ (modulo logical equivalence). An $\mathcal{L}(A)$-formula is a formula with parameters only from $A$. An $\mathcal{L}$-formula is a formula over the empty set. If $\bar{a} = a_1,\ldots,a_n$ is a sequence of points from $\mathcal{U}$ and $A$ is a subset of $\mathcal{U}$, we let $\tp_{\qf}(\bar{a}/A)$ denote the quantifier-free type of $\bar{a}$ over $A$. If $\theta(x_1,\ldots,x_n)$ is a $\mathcal{L}(\mathcal{U})$-formula, we may sometimes write the same formula as $\theta(x_{n};x_1,\ldots,x_{n-1})$ to place special emphasis on a particular variable, in this case $x_n$.

If $A \subseteq \mathcal{U}$, we let $S_{x}(A)$ be the associated type space. The central type space of study will be $S_{\mathbf{x}}(\mathcal{U})$ where $\mathbf{x} = (x_i)^{\omega}_{i \geq 1}$. We recall that $S_{\mathbf{x}}(\mathcal{U})$ is a Stone space, i.e.\ compact, Hausdorff, and totally disconnected. For indices $i_1,\ldots,i_n$, we will often consider the set, 
\begin{equation*}
    [\varphi(x_{i_1},\ldots,x_{i_n})] \coloneqq \{p \in S_{\mathbf{x}}(\mathcal{U}) : \varphi(x_{i_1},\ldots,x_{i_n}) \in p\}. 
\end{equation*}
These sets form a clopen basis for $S_{\mathbf{x}}(\mathcal{U})$ and in particular, they are Borel. We will often be interested in the space of non-redundant types such that none of coordinates are realized. We denote this space as $S_{\mathbf{x}}^{+}(\mathcal{U})$. More explicitly, 
\begin{equation*}
    S_{\mathbf{x}}^{+}(\mathcal{U}) := \{p \in S_{\mathbf{x}}(\mathcal{U}) : (\forall i \neq j)\  (x_i \neq x_j) \in p \text{~and~} (\forall a \in \mathcal{U})\ (x_i \neq a) \in p\}. 
\end{equation*}

We let $\mathfrak{M}_{x}(A)$ be the collection of Keisler measures on $\mathcal{L}_{x}(A)$, i.e.\ finitely additive probability measures on $\mathcal{L}_{x}(A)$. We recall the fact that every Keisler measure on $\mathcal{L}_{x}(A)$ extends uniquely to a countably additive regular Borel probability measure on $S_{x}(A)$, even when $x$ is replaced by an infinite tuple of variables. More explicitly, the following statement is from Fremlin's measure theory treatise \cite[416Q Proposition (b)]{fremlin2000measure} restricted to the probability measure context. 

\begin{fact}\label{fact:stone} Let $\mathbb{B}$ be a Boolean algebra and $S(\mathbb{B})$ its Stone space. Then there is a one-to-one correspondence between finitely additive probability measures on $\mathbb{B}$ and regular Borel probability measures on $S(\mathbb{B})$ given by $\mu(A) = \tilde{\mu}([A])$ where for each $A \in \mathbb{B}$, $[A]$ is the corresponding clopen subset of $S(\mathbb{B})$. 
\end{fact}

We often make use of this correspondence without comment. The following is a \emph{Keisler measure cheat sheet} which is included for the reader's convenience. We remark that most (if not all) of the following definitions are originally from \cite{NIP1,NIP2,NIP3}. 

\begin{definition}\label{cheat} Let $\mu \in \mathfrak{M}_{\bar{x}}(\mathcal{U})$ and $M$ be a small elementary submodel. 
\begin{enumerate}
    \item We let $\mathfrak{M}_{\bar{x}}^{\inv}(\mathcal{U},M)$ be the collection of measures in $\mathfrak{M}_{\bar{x}}(\mathcal{U})$ which are $M$-invariant. We recall that $\mu$ is $M$-invariant if for any $\bar{a}, \bar{b} \in \mathcal{U}^{\bar{y}}$ such that $\bar{a} \equiv_{M} \bar{b}$, we have that $\mu(\varphi(\bar{x},\bar{a})) = \mu(\varphi(\bar{x},\bar{b}))$ for any $\Lc$-formula $\varphi(\bar{x},\bar{y})$. 
    \item Suppose $\mu \in \mathfrak{M}^{\inv}_{\bar{x}}(\mathcal{U},M)$. Then for any $\Lc$-formula $\varphi(\bar{x},\bar{y})$, we can define the map $F_{\mu,M}^{\varphi}: S_{\bar{y}}(M) \to [0,1]$ via $F_{\mu,M}^{\varphi}(q) = \mu(\varphi(\bar{x},\bar{b}))$ where $\bar{b} \models q$. We remark that this map is well-defined since $\mu$ is $M$-invariant. 
    \item Suppose $\mu \in \mathfrak{M}^{\inv}_{\bar{x}}(\mathcal{U},M)$. We say that $\mu$ is Borel-definable if for every $\mathcal{L}$-formula $\varphi(\bar{x},\bar{y})$, the map $F_{\mu,M}^{\varphi}$ is a Borel function. 
    \item Suppose $\mu \in \mathfrak{M}^{\inv}_{\bar{x}}(\mathcal{U},M)$. We say that $\mu$ is definable if for every $\mathcal{L}$-formula $\varphi(\bar{x},\bar{y})$, the map $F_{\mu,M}^{\varphi}$ is a continuous function. 
    \item Suppose that $\mu \in \mathfrak{M}^{\inv}_{\bar{x}}(\mathcal{U},M)$ and $\nu \in \mathfrak{M}_{\bar{x}}(\mathcal{U})$. Moreover, suppose that $\mu$ is Borel-definable. We define the Morley product $\mu \otimes \nu$ as follows: For any formula $\varphi(\bar{x},\bar{y}) \in \mathcal{L}_{\bar{x} \bar{y}}(\mathcal{U})$, we have that 
    \begin{equation*}
        (\mu \otimes \nu)(\varphi(\bar{x},\bar{y})) = \int_{S_{\bar{y}}(M')} F_{\mu,M'}^{\varphi} d\nu|_{M'},
    \end{equation*}
    where $M'$ is any small model containing $M$ and all the parameters from $\varphi$. The measure $\nu|_{M'}$ is the regular Borel probability measure corresponding to the restriction of $\nu$ to $\mathcal{L}_{\bar{y}}(M')$. We remark that this product is well-defined. In practice, we often drop the $M'$ from the notation when there is no possibility of confusion, e.g.\ $F_{\mu}^{\varphi}$ instead of $F_{\mu,M'}^{\varphi}$ and $\nu$ instead of $\nu|_{M'}$.
    \item The Morley product is a \emph{separated amalgam}. If $\bar{x}$ and $\bar{y}$ are distinct tuples of variables, $\mu \in \mathfrak{M}_{\bar{x}}^{\inv}(\mathcal{U},M)$ is Borel-definable, and $\nu \in \mathfrak{M}_{\bar{y}}(\mathcal{U})$ then for any $\varphi(\bar{x}) \in \mathcal{L}_{\bar{x}}(\mathcal{U})$ and $\psi(\bar{y}) \in \mathcal{L}_{\bar{y}}(\mathcal{U})$, we have that 
    \begin{equation*}
        (\mu \otimes \nu)(\varphi(\bar{x}) \wedge \psi(\bar{y})) = \mu(\varphi(\bar{x})) \cdot \nu(\psi(\bar{y})). 
    \end{equation*}
    See \cite{NIP3}. 
    \item  Suppose $\mu \in \mathfrak{M}_{\bar{x}}^{\inv}(\mathcal{U},M)$. We say that $\mu$ is \emph{generically stable} or \emph{fim} (\emph{frequency interpretation measure}) over $M$  if $\mu$ is definable and for any $\mathcal{L}$-formula $\varphi(\bar{x},\bar{y})$ there exists a sequences of $\mathcal{L}(M)$-formula $(\theta_{n}(\bar{x}_1,\ldots,\bar{x}_n))_{n \geq 1}$ where $|\bar{x}_i| = |\bar{x}|$ and for $i \neq j$, $\bar{x}_i \cap \bar{x}_j = \emptyset$ such that
\begin{enumerate}
    \item for any $\epsilon > 0$ there exists some $n_{\epsilon}$ such that if $n > n_{\epsilon}$ and $\mathcal{U} \models \theta_{n}(\bar{a}_1,\ldots,\bar{a}_n)$ then 
    \begin{equation*}
        \sup_{\bar{b} \in \mathcal{U}^{\bar{y}}}|\mu(\varphi(\bar{x},\bar{b})) - \Av(\bar{a}_1,\ldots,\bar{a}_n)(\varphi(\bar{x},\bar{b}))| < \epsilon,
    \end{equation*}
    where $\Av(\bar{a}_1,\ldots,\bar{a}_n) = \frac{1}{n}\sum_{i=1}^{n}\delta_{\tp(\bar{a}_i/\mathcal{U})}$. 
    \item $\lim_{n \to \infty} \mu^{(n)}(\theta_{n}(\bar{x}_1,\ldots,\bar{x}_n)) = 1$.  
\end{enumerate}
\item Suppose that $\mu \in \mathfrak{M}_{\bar{x}}^{\inv}(\mathcal{U},M)$. We say that $\mu$ is smooth over $M$ if whenever $\nu|_{M} = \mu|_{M}$, we may conclude that $\nu = \mu$. In other words, $\mu$ is the unique extension of $\mu|_{M}$ to the monster model. 
\end{enumerate}
\end{definition}

The first implication in the following fact is \cite[Corollary 2.6]{NIP3}. The second implication is straightforward; it follows directly from the fact that uniform limits of continuous functions are continuous.

\begin{fact} Suppose that $\mu \in \mathfrak{M}_{\bar{x}}^{\inv}(\mathcal{U},M)$. Then
\begin{enumerate}
\item If $\mu$ is smooth over $M$, then $\mu$ is \emph{fim} over $M$.
\item If $\mu$ is \emph{fim} over $M$, then $\mu$ is definable over $M$.
\end{enumerate}
In other words, smoothness implies generic stability, which in turn implies definability.
\end{fact}

The following definitions of \emph{adequate} and \emph{excellent} are new. The motivation behind the definition is obvious; we want to define a class of measures which interacts nicely with the Morley product, but we do not want to be constrained to the NIP setting. The terms \emph{\weak} and \emph{\strong} will become relevant in later sections when we define the \emph{\weak\ measure extension axiom} and the \emph{\strong\ measure extension axiom}.   

\begin{definition} Let $\mu \in \mathfrak{M}^{\inv}_{\bar{x}}(\mathcal{U},M)$. We say that $\mu$ is $M$-\emph{\weak} if 
\begin{enumerate}
    \item $\mu$ is Borel-definable over $M$. 
    \item Any iteration of the Morley product of $\mu$ is Borel-definable over $M$, i.e.\ for any $n \geq 1$, the measure $\mu^{(n+1)}(\bar{x}_1,\ldots,\bar{x}_{n+1}) = \mu(\bar{x}_{n+1}) \otimes \mu^{(n)}(\bar{x}_1,\ldots,\bar{x}_n)$ is Borel-definable over $M$. 
    \item $\mu$ is self-associative, i.e.\ the measure $\bigotimes_{i=1}^{n} \mu(\bar{x}_i)$ does not depend on the placement of parentheses. 
\end{enumerate}
We say that $\mu$ is $M$-\strong\ if $\mu$ is $M$-\weak\ and $\mu$ is self-commuting, i.e.\  $\mu_{\bar{x}} \otimes \mu_{\bar{y}} = \mu_{\bar{y}} \otimes \mu_{\bar{x}}$. Obviously all $M$-\strong\ measures are $M$-\weak. 
\end{definition}

\noindent One should keep the following examples of measures in mind. 

\begin{example}\label{example:main} The following are examples of \strong\ measures. 
\begin{enumerate}
    \item Suppose that $T$ is NIP. Then $\mu \in \mathfrak{M}_{\bar{x}}^{\inv}(\mathcal{U},M)$ is $M$-\strong\ if and only if $\mu$ is generically stable over $M$ (self-commuting is equivalent to generic stability; see \cite{NIP3}. Self-associativity follows from, e.g., \cite{CG,GanCon2}). We will consider the following concrete example.  Let $M = (\mathbb{R},<)$ and let $L$ be the standard Lebesgue measure on $\mathbb{R}$ restricted to the interval $[0,1]$. We define the Keisler measure $\mu_{L}$ on $\mathcal{L}_{x}(\mathcal{U})$ as follows: For any $\varphi(x) \in \mathcal{L}_{x}(\mathcal{U})$, 
\begin{equation*}
    \mu_{L}(\varphi(x)) = L(\{ r \in \mathbb{R}: \mathcal{U} \models \varphi(r)\}). 
\end{equation*}
Since the structure is $o$-minimal, every definable set is a finite union of points and intervals. In particular, the sets on the RHS are Borel subsets of $\mathbb{R}$ and so $\mu_{L}$ is well-defined. One can prove that $\mu_{L}$ is in $\mathfrak{M}_{x}^{\inv}(\mathcal{U},M)$ and is generically stable, smooth even. In particular, $\mu_{L}$ is $M$-\strong. 
    \item Let $\mathcal{L} = \{R\}$ and $T$ be the theory of the Rado graph.  For $t \in (0,1)$, we let $\mu_t$ be the unique Keisler measure in $\mathfrak{M}_{x}(\mathcal{U})$ such that for any distinct sequence of tuples $a_1,\ldots,a_n,b_1,\ldots,b_m$, we have that
\begin{equation*}
    \mu_t \left(\bigwedge_{i=1}^{n} R(x,a_i) \wedge \bigwedge_{j = 1}^{m}  \neg R(x,b_j) \right) =  t^{n} \left( 1 - t \right)^{m}. 
\end{equation*}
For any small model $M$, we claim that $\mu_t$ is $M$-\strong\ ((self-associativity follows from the associativity of the Morley product for definable measures; see \cite{CG}. Self-commuting is checked below).). In particular, for any $t \in (0,1)$, $\mu_{t}$ is $M$-\strong. 
\item In general, if $\mu \in \mathfrak{M}_{\bar{x}}^{\inv}(\mathcal{U},M)$ and $\mu$ is \emph{fim} over $M$, then $\mu$ is $M$-excellent. Self-associativity follows from associativity of definable measures while a proof of commuting with all Borel-definable measures (and thus self-commuting) can be found in \cite{CGH}.  
\end{enumerate}
Additionally, if $T$ is NIP, then a measure $\mu \in \mathfrak{M}_{\bar{x}}(\mathcal{U})$ is $M$-\weak\ if and only if $\mu$ is $M$-invariant (see e.g., \cite{NIP2,NIP3,GanCon2}). In general,  all $M$-definable measures are $M$-adequate. 
\end{example}

\begin{fact} For any $t \in (0,1)$, the measure $\mu_{t}$ in Example \ref{example:main} self-commutes. 
\end{fact}

\begin{proof} By quantifier elimination and standard facts about measures, it suffices to show that $\mu_{t,x} \otimes \mu_{t,y}$ and $\mu_{t,y} \otimes \mu_{t,x}$ agree on intersections of literals. This follows from the fact that if two probability measures agree on a $\pi$-system, then they agree on the $\sigma$-algebra generated by that $\pi$-system. We claim moreover that it suffices to show these measures agree on formulas of the form

\begin{equation*}
    \Psi(x,y) \coloneqq \bigwedge_{t \in A} R^{\epsilon(t)}(x,a_t) \wedge \bigwedge_{s \in B} R^{\epsilon(s)}(y,b_s) \wedge \bigwedge_{w \in C} x \neq c_{w} \wedge \bigwedge_{z \in D} y \neq d_{z} \wedge P(x,y),
\end{equation*}
where $P(x,y) \in \{R(x,y),(\neg R(x,y) \wedge x\neq y),x=y,x\neq y\}$ and $\epsilon\colon A \sqcup B \to \{1,0\}$ with the convention that $R^{1}(-) = R(-)$ while $R^{0}(-) = \neg R$. Choose a small model $N$ containing $M$ and all the parameters in our formula $\Psi$. Let $\varphi_1(y) \coloneqq \bigwedge_{s \in B} R^{\epsilon(s)}(y,b_{s}) \wedge \bigwedge_{z \in D} y \neq d_{z}$, $\varphi_2(x) \coloneqq \bigwedge_{t \in A} R^{\epsilon(t)}(x,a_{t}) \wedge \bigwedge_{w \in C} x \neq c_{w}$, and $p = \mu_t(P(x,y))$. We remark that the possible values for $p$ are $t$, $1-t$, $0$, and $1$. Then
\begin{align*}
    (\mu_{t,x} \otimes \mu_{t,y})(\Psi(x,y)) &= \int_{S_{y}(N)} F_{\mu_{t,x}}^{\Psi} d\mu_{t,y} \\ &= \int_{S_{y}(N)} \mathbf{1}_{\varphi_1(y)} \cdot p \cdot \mu_{t,x}(\varphi_2(x)) d\mu_{t,y} \\ &= p \cdot \mu_{t,x}(\varphi_2(x)) \cdot \mu_{t,y}(\varphi_1(y)).
\end{align*}
An almost identical computation gives $(\mu_{t,y} \otimes \mu_{t,x})(\Psi(x,y)) = p \cdot \mu_{t,x}(\varphi_2(x)) \cdot \mu_{t,y}(\varphi_1(y))$. \qedhere
\end{proof}

\subsection{Supports} Here we recall some basic facts relating to the supports of measures. We recall that \emph{invariantly supported measures} were defined in \cite{chernikov2022definable}. 

\begin{definition} Let $A \subseteq \mathcal{U}$ and $\mu \in \mathfrak{M}_{\bar{x}}(A)$. We let $\supp(\mu)$ denote the support of $\mu$, i.e.
\begin{equation*}
    \supp(\mu) \coloneqq \{p \in S_{\bar{x}}(A): \mu(\varphi(\bar{x})) > 0, \text{for any } \varphi(\bar{x}) \in p\}. 
\end{equation*}
This definition makes sense when $\bar{x}$ is an infinite tuple of variables. Moreover we say that a measure $\mu$ is \emph{invariantly supported} (over $M$) if $\mu$ is $M$-invariant and for every $p \in \supp(\mu)$, we have that $p$ is $M$-invariant. 
\end{definition}

Recall the measures from Example \ref{example:main}. We remark that $\mu_{L}$ is invariantly supported while for any $t \in (0,1)$, $\mu_{t}$ is not invariantly supported (see  \cite{chernikov2022definable} for discussion). The next fact recalls that all invariant measures in NIP theories are invariantly supported. See \cite{Guide} for the definition of the Morley product of invariant types. Formally speaking, the Morley product for invariant types is defined differently than the Morley product for Keisler measures. This is because there are no Borel definability hypotheses which need to be fulfilled to define the Morley product of invariant types. However, if $p$ is Borel-definable and $q$ is any type, then $\delta_{p \otimes q} = \delta_{p} \otimes \delta_{q}$ (see e.g., \cite[Proposition 6.5]{gannon2020approximation}).

\begin{fact}\label{fact:support1} Suppose that $\mu \in \mathfrak{M}_{\bar{x}}(\mathcal{U})$ and let $M$ be a small submodel of $\mathcal{U}$. 
\begin{enumerate}
    \item $\supp(\mu)$ is a closed non-empty subset of $S_{\bar{x}}(\mathcal{U})$. This is still true when $\bar{x}$ is replaced by an infinite tuple of variables.
    \item (T NIP) If $\mu$ is $M$-invariant, then $\mu$ is invariantly supported \cite[Prop.~7.15]{Guide}. 
\end{enumerate}
\end{fact}

\begin{definition} Let $\mu  \in \mathfrak{M}_{\bar{x}}^{\inv}(\mathcal{U},M)$ and $\nu \in \mathfrak{M}_{\bar{y}}(\mathcal{U})$. We define the notation:
\begin{equation*}
    \supp(\mu) \otimes \supp(\nu) := \{p \otimes q: p \in \supp(\mu), q \in \supp(\nu)\}. 
\end{equation*}
\end{definition}

\begin{proposition}\label{prop:support} Let $\mu \in \mathfrak{M}^{\inv}_{\bar{x}}(\mathcal{U},M)$ and $\nu \in \mathfrak{M}_{\bar{y}}(\mathcal{U})$. 
\begin{enumerate}
    \item If $\mu$ is Borel-definable, then for every $p \in \supp(\mu)$ and $q \in \supp(\nu)$, there exists some $r \in \supp(\mu \otimes \nu)$ such that $r|_{\bar{x}} = p$ and $r|_{\bar{y}} = q$. 
    \item If $\mu$ is Borel-definable and invariantly supported, then $(\supp(\mu) \otimes \supp(\nu)) \cap \supp(\mu \otimes \nu)$ is a dense subset of $\supp(\mu \otimes \nu)$. 
    \item If $\mu$ is definable and invariantly supported, then $\supp(\mu) \otimes \supp(\nu) \subseteq \supp(\mu \otimes \nu)$. 
\end{enumerate}
\end{proposition}

\begin{proof} 
        For (1), notice that for any $\varphi(\bar{x}) \in p$ and $\psi(\bar{y}) \in q$, we have that
        \begin{equation*}
            (\mu \otimes \nu)(\varphi(\bar{x}) \wedge \psi(\bar{y})) = \mu(\varphi(\bar{x}))\nu(\psi(\bar{y})) > 0. 
        \end{equation*}
        By \cite[Proposition 2.7]{chernikov2022definable}, $\supp(\mu \otimes \nu) \cap [\varphi(\bar{x}) \wedge \psi(\bar{y})] \neq \emptyset$. If we let $A_{\varphi,\psi} = \supp(\mu \otimes \nu) \cap [\varphi(\bar{x}) \wedge \psi(\bar{y})]$, we notice that family $\mathcal{A} \coloneqq (A_{\varphi,\psi})_{\varphi \in p, \psi \in q}$ has the finite intersection property. Thus $\bigcap \mathcal{A}$ is non-empty. Take $r \in \bigcap \mathcal{A}$. 
        
        For (2), fix $\theta(\bar{x},\bar{y}) \in \mathcal{L}_{\bar{x}\bar{y}}(\mathcal{U})$. Suppose that $[\theta(\bar{x},\bar{y})] \cap \supp(\mu \otimes \nu) \neq \emptyset$. Then there exists some $r \in \supp(\mu\otimes \nu)$ such that $\theta(\bar{x},\bar{y}) \in r$. Since $r$ is in the support, $(\mu \otimes \nu)(\theta(\bar{x},\bar{y})) > 0$. Choose a model $M'$ such that $M'$ contains $M$ as well as all the parameters from $\theta$. Then $\int_{S_{\bar{y}}(M')} F_{\mu,M'}^{\theta}d\nu > 0$ and so there exists some $q \in \supp(\nu)$ such that $F_{\mu,M'}^{\theta}(q|_{M'}) > 0$. So $\mu(\theta(\bar{x},\bar{b})) > 0$ where $\bar{b}\models q|_{M'}$. By \cite[Proposition 2.7]{chernikov2022definable}, there exists some $p \in \supp(\mu)$ such that $\theta(\bar{x},\bar{b}) \in p$. Then we have that $\theta(\bar{x},\bar{y}) \in p \otimes q$.  
        
         For (3), note that for each $p \in \supp(\mu)$ and $q \in \supp(\nu)$, the Morley product $p \otimes q$ is well-defined. Now fix $p \in \supp(\mu)$, $q \in \supp(\nu)$, and $\theta(\bar{x},\bar{y}) \in p \otimes q$. Let $N$ be a small model such that $M \subseteq N$ and $N$ contains all the parameters from $\theta$. Then $\theta(\bar{x},\bar{b}) \in p$ where $\bar{b} \models q|_{N}$. Hence $\mu(\theta(\bar{x},\bar{b})) > 0$ since $p \in \supp(\mu)$. Then $F_{\mu,N}^{\theta}(p|_N)>0$ by definition. Since $\mu$ is definable over $M$, it is also definable over $N$, and so the map $F_{\mu,N}^{\theta}$ is continuous. Therefore $\int_{S_{\bar{y}}(N)} F_{\mu,N}^{\theta} d\nu > 0$. And by definition, $(\mu \otimes \nu)(\theta(\bar{x},\bar{y})) > 0$. \qedhere
\end{proof}

\subsection{Space of labeled $\mathcal{L}$-structures}

We will briefly discuss the relation between generic sampling and the space of labeled $\mathcal{L}$-structures. This on going research (e.g., \cite{ackerman2025generic}). Here, we require our labels to be strictly positive integers so that we do not have an off-by-one error when connecting this space with generic sampling.

\begin{definition} Fix a finite relational language $\mathcal{L}$. Then the space of labeled $\mathcal{L}$-structures $\str_{\mathcal{L}}$ is the space of all $\mathcal{L}$-structures with underlying domain $\mathbb{N}_{>0}$. For each $\mathcal{L}$-formula $\varphi(x_1,\ldots,x_{n})$ and natural numbers $i_1,\ldots,i_{n}$, we let 
\begin{equation*}
    \llbracket \varphi(i_1,\ldots,i_{n}) \rrbracket = \{M \in \str_{\mathcal{L}} : M \models \varphi(i_1,\ldots,i_{n}) \}.
\end{equation*}
We recall that $\str_{\mathcal{L}}$ is a compact Hausdorff space with a basis of open sets given by 
\begin{equation*}
    \{\llbracket \varphi(i_1,\ldots,i_{n}) \rrbracket : \varphi(x,\ldots,x_{n}) \text{ is a quantifier-free $\mathcal{L}$-formula, }i_1,\ldots,i_{n} \in \mathbb{N}_{>0}\}. 
\end{equation*}
\end{definition}

\begin{fact}\label{fact:invariant-measures-agree}
Let $\lambda$ and $\eta$ be Borel probability measures on $\operatorname{Str}_{\mathcal{L}}$. Then $\lambda = \eta$ if and only if for every quantifier-free $\mathcal{L}$-formula $\theta(x_1,\dots,x_n)$ and every tuple of distinct positive integers $i_1,\dots,i_n$, we have
\[
\lambda\bigl(\llbracket \theta(i_1,\dots,i_n)\rrbracket\bigr) = \eta\bigl(\llbracket \theta(i_1,\dots,i_n)\rrbracket\bigr).
\]
\end{fact}
\section{Generic sampling}

We begin by defining the probability space that is central to this paper. We remark that the space itself is not \emph{new} and has been implicitly studied in some situations (see e.g., \cite{NIP3,conant2025generic}). However, it is usually framed through the eyes of model theory, i.e.\ when one considers the iterated Morley  product, we typically tend to treat it as a \emph{Morley sequence in a measure}. Here, however, we want to treat it as a probability space \emph{in itself}. We recall that $\mathbf{x} = (x_i)_{i \geq 1}^{\omega}$. 

\begin{definition} Let $\mu \in \mathfrak{M}^{\inv}_{x}(\mathcal{U},M)$ and suppose that $\mu$ is Borel-definable. We define the measure $\mathbb{P}_{\mu}$ on $S_{\mathbf{x}}(\mathcal{U})$ as follows: 
\begin{enumerate}
    \item $\mu^{(1)}(x_1)= \mu(x_1)$. 
    \item $\mu^{(n+1)}(x_1,\ldots,x_{n+1}) = \mu(x_{n+1}) \otimes \mu^{(n)}(x_1,\ldots,x_n)$. 
    \item $\mathbb{P}_{\mu} = \bigcup_{i=1}^{\omega} \mu^{(n)}(x_1,\ldots,x_{n+1})$. 
\end{enumerate}  The probability space we care about is $(S_{\mathbf{x}}(\mathcal{U}), \mathcal{B} ,\widehat{\mathbb{P}_{\mu}})$ where $\mathcal{B}$ is the Borel $\sigma$-algebra of $S_{\mathbf{x}}(\mathcal{U})$ and $\widehat{\mathbb{P}_{\mu}}$ is the unique regular Borel probability measure on $S_{\mathbf{x}}(\mathcal{U})$ such that for every $\mathcal{L}_{\mathbf{x}}(\mathcal{U})$-formula $\varphi(x_1,\ldots,x_n)$, we have that 
\begin{equation*}
    \widehat{\mathbb{P}_\mu}([\varphi(x_1,\ldots,x_n)]) = \mathbb{P}_{\mu}(\varphi(x_1,\ldots,x_n)). 
\end{equation*}
\end{definition}

\begin{remark} Intuitively, one should think of the space $(S_{\mathbf{x}}(\mathcal{U}),\mathcal{B},\mathbb{P}_{\mu})$ as an object arising from a \emph{random process}. The following explanation is not technically correct, but as Pillay would say, it is \emph{morally correct}: At each stage, we randomly and independently choose a point in $\supp(\mu)$ and then concatenate the Morley product to get an element in $S_{\mathbf{x}}(\mathcal{U})$. Then for any $A \in \mathcal{B}$, $\mathbb{P}_{\mu}(A)$ can be thought of as the likelihood that a type constructed in this way is in $A$. This is precisely what is occurring when our measure $\mu$ is a sum of finitely many invariant types, i.e. $\mu = \sum_{i=1}^{n} r_i \delta_{p_i}$ where each $p_i \in S_{x}^{\inv}(\mathcal{U},M)$, $r_i \geq 0$ and $\sum_{i=1}^{n} r_i = 1$ (see (3) of Proposition \ref{prop:bigsupport}). The situation is more complicated in the general case.
\begin{enumerate}
    \item It is possible for $\supp(\mathbb{P}_{\mu})$ to contain elements which do not arise simply as an iterated Morley product, i.e. $q \in \supp(\mathbb{P}_{\mu})$ and $q \neq \bigotimes_{i=1}^{\omega} p_i(x_i)$ for $p_i \in \supp(\mu)$. Hence the gluing together of types in-between stages is more complex than just taking the Morley product. Moreover, we only know that $\{\bigotimes_{i=1}^{\omega} p_i(x_i) : p_i \in \supp(\mu)\} \subseteq \supp(\mathbb{P}_{\mu})$ when $\mu$ is definable and invariantly supported (see (2) of Proposition \ref{prop:bigsupport}). 
    \item It is also possible that almost none of the types in the support of $\mu$ are invariant and so the Morley product does not even make sense on elements of the support. This happens only outside of the NIP context, but it does occur with excellent measures in the Rado graph. Indeed, for $\mu_{t}$ in Example \ref{example:main}, $\supp(\mu_{t}) = S_{x}(\mathcal{U}) \backslash \{\tp(a/\mathcal{U}): a \in \mathcal{U}\}$. 
\end{enumerate}
Finally, we remark that Proposition \ref{prop:random} and the paragraph after it argue that the random process associated to $(S_{\mathbf{x}}(\mathcal{U}),\mathcal{B},\mathbb{P}_{\mu})$ extends the standard random process of sampling points from $S_{x}(\mathcal{U})$ with respect to $\mu$. 
\end{remark}

There are some quirks that the uninitiated reader should keep in mind. First, the Morley product is often not ``commutative'' and sometimes not ``associative'' with respect to formulas from $\mathcal{L}_{\mathbf{x}}(\mathcal{U})$ (see, i.e., \cite{CGH}). For instance, it is common that the following inequality holds:\footnote{The sets we are measuring on either side of the inequality are typically different.}
\begin{equation*}
    \mathbb{P}_{\mu}\left([\varphi(x_1,x_2)]\right) \neq \mathbb{P}_{\mu}\left([\varphi(x_2,x_1)]\right).
\end{equation*}
Hence, when our measure $\mu$ is only Borel-definable or $M$-\weak, one needs to pay attention to the order of the variables. 

Secondly, the space $S_{x}(\mathcal{U})$ is big and by extension, so is $S_{\mathbf{x}}(\mathcal{U})$. In particular, while $S_{\mathbf{x}}(\mathcal{U})$ is a compact Hausdorff space, it is not Polish. However, this actually seems to cause less of a problem than one might imagine. Many of the events we care about as well as those that model theorists should care about are still Borel relative to $S_{\mathbf{x}}(\mathcal{U})$ and so they are $\mathbb{P}_{\mu}$-measurable (for example, see Proposition \ref{prop:borel}).

Finally, we should compare and contrast the spaces $(S_{\mathbf{x}}(\mathcal{U}),\mathcal{B},\mathbb{P}_{\mu})$ with the \emph{standard sampling space}. More explicitly,  the \emph{product measure on the cartesian product of our space} $(\prod_{i\geq 1}^{\omega} S_{x_i}(\mathcal{U}), \mathcal{D}, \prod_{i\geq 1}^\omega \mu(x_i))$ where $\mathcal{D}$ is Borel $\sigma$-algebra on the product. 
As a convention, we will drop the indices off the product space. 
Let
$\prod \mu$ be the unique product measure on $(\prod S_x(\mathcal{U}), \mathcal{D})$. 

We remark that the probability space we are studying is an extension of the standard sampling space. The following fact brings this claim into focus.  

\begin{proposition}\label{prop:random} Let $\mu \in \mathfrak{M}_{x}^{\inv}(\mathcal{U},M)$ such that $\mu$ is Borel-definable. Let $r:S_{\mathbf{x}}(\mathcal{U}) \to \prod S_{x}(\mathcal{U})$ be the restriction map given by $r(p) = (p|_{x_i})_{i \geq 1}^{\omega}$. Then the following hold: 
\begin{enumerate}
    \item $r$ is a continuous surjection between compact Hausdorff spaces. Hence $r$ is a quotient map. 
    \item The pushforward of $\mathbb{P}_{\mu}$ is precisely $\prod \mu$, i.e.\ $r_{*}(\mathbb{P}_{\mu}) = \prod \mu$. 
    \item By the previous observation, if an event occurs for $\prod \mu$-almost all $(p_i)_{i \geq 1}^{\omega} \in \prod S_{x}(\mathcal{U})$, then the corresponding event occurs for $\mathbb{P}_{\mu}$-almost all $p \in S_{\mathbf{x}}(\mathcal{U})$. This is just a fancy way to say that if $B \in \mathcal{D}$ and $\prod \mu(B) = 1$ then $\mathbb{P}_{\mu}(r^{-1}(B)) = 1$. 
\end{enumerate}
\end{proposition}

\begin{proof} We prove the statements: 
\begin{enumerate}
    \item Clear. 
    \item Since the collection of basic open subsets of $\prod S_{x}(\mathcal{U})$ forms a $\pi$-system which generates $\mathcal{D}$, it suffices to show that the measures $r_{*}(\mathbb{P}_{\mu})$ and $\prod \mu$ agree on basic open subsets of $\prod S_{x}(\mathcal{U})$. First assume that $A_i = [\theta_i(x_i)]$ for $i \leq n$. Then 
    \begin{align*}
        r_{*}(\mathbb{P}_{\mu}) \left(\prod A_i \right) &= \mathbb{P}_{\mu}\left( \left\{p \in S_{\mathbf{x}}(\mathcal{U}) : \bigwedge_{i=1}^{n} \theta_i(x_i) \in p \right\} \right) \\ &= \mathbb{P}_{\mu}\left( \left[ \bigwedge_{i=1}^{n} \theta_i(x_i) \right] \right) \\ &=\mu^{(n)}\left( \bigwedge_{i=1}^{n} \theta_i(x_i) \right) \\ &= \prod_{i=1}^{n} \mu_{x_i}(\theta_i(x_i)) \\ &= \prod \mu \left(\prod A_i \right)
    \end{align*}
where the last equation follows from induction and the fact that the Morley product is a separated amalgam. Now assume that $A_i$ is open for $i \leq n$. Then $A_i = \bigcup_{j_i \in J_i} [\theta_{j_i}(x_i)]$. For each finite collection of indices $\bar{j}_i = \{j_{i_1},\ldots,j_{i_k}\}$, we let $\psi_{\bar{j}_i}(x_i) = [\theta_{j_{i_1}}(x_i) \vee \ldots \vee \theta_{j_{i_k}}(x_i)]$. Then we can write $\prod A_i$ as an upward directed family of open sets, namely 
\begin{equation*}
    \prod A_i = \bigcup_{(\bar{j}_1,\ldots,\bar{j}_n) \in \mathcal{P}_{\fin}(J_1) \times \ldots \times \mathcal{P}_{\fin}(J_n)} [\psi_{\bar{j}_{1}}(x_1) \wedge \ldots \wedge \psi_{\bar{j}_{n}}(x_n) ],
\end{equation*}
where $\mathcal{P}_{\fin}(J_{i})$ is the collection of all finite subsets of $J_{i}$ and the elements of the index set are ordered by coordinate-wise inclusion. Then, 
 \begin{align*}
        r_{*}(\mathbb{P}_{\mu}) \left(\prod A_i \right) &= \mathbb{P}_{\mu}\left( \bigcup_{(\bar{j}_1,\ldots,\bar{j}_n) \in \mathcal{P}_{\fin}(J_1) \times \ldots \times \mathcal{P}_{\fin}(J_n)} [\psi_{\bar{j}_1}(x_1) \wedge \ldots \wedge \psi_{\bar{j}_n}(x_n)] \right) \\ &\overset{(*)}{=} \sup_{(\bar{j}_1,\ldots,\bar{j}_n) \in \mathcal{P}_{\fin}(J_1) \times \ldots \times \mathcal{P}_{\fin}(J_n)} \mathbb{P}_{\mu}\left( \left[\bigwedge_{i=1}^{n} \psi_{\bar{j}_i}(x_i) \right] \right) \\ & \overset{(\dagger)}{=} \sup_{(\bar{j}_1,\ldots,\bar{j}_n) \in \mathcal{P}_{\fin}(J_1) \times \ldots \times \mathcal{P}_{\fin}(J_n)} \prod \mu_{x_i}(\psi_{\bar{j}_i}(x_i)) \\ &= \prod \sup_{\bar{j}_i \in \mathcal{P}_{\fin}(J_i)} \mu_{x_i}(\psi_{\bar{j}_i}(x_i)) \\ &\overset{(**)}{=} \prod \mu_{x_i}(A_i) \\ &= \prod\mu\left( \prod A_i \right). 
\end{align*}
Equations $(*)$ and $(**)$ follow from $\tau$-additivity of our measures $\mathbb{P}_{\mu}$ and $\mu$, respectively (both are regular Borel probability measures on compact Hausdorff spaces). Equation $(\dagger)$ follows from the previous case. 
\item Clear. \qedhere
\end{enumerate}
\end{proof}

Aided by the previous fact, we have the following alternative interpretation in terms of random processes. A random point in the space $\prod S_{x}(\mathcal{U})$ corresponds to infinitely many independent trials. In each trial, we choose a point randomly with respect to $\mu$. To find a random point in $S_{\mathbf{x}}(\mathcal{U})$, we do the same thing, but then must make decisions about more formulas, e.g.\ if we choose $p_{x_1}$ and then $p_{x_2}$, we need to make some choice about formulas of the form $\varphi(x_1,x_2)$. These choices are determined generically by the Morley product. 

We use the following jargon without excuse. 

\begin{definition} Given a Borel-definable measure $\mu \in \mathfrak{M}_{x}^{\inv}(\mathcal{U},M)$, and a subset $B \subseteq S_{\mathbf{x}}(\mathcal{U})$ we say that:
\begin{enumerate}
    \item $B$ holds on \emph{almost all} random generic types if $\mathbb{P}_{\mu}(B) = 1$. 
    \item $B$ holds on \emph{almost no} random generic types if $\mathbb{P}_{\mu}(B) = 0$. 
\end{enumerate}
\end{definition}

\subsection{The support} The support of the measure $\mathbb{P}_{\mu}$ is an interesting set in itself. The elements in $\supp(\mathbb{P}_{\mu})$ can be thought of as those types in \emph{which one could potentially sample}. Let's make a few observations in the invariantly supported context.

\begin{proposition}\label{prop:bigsupport} Let $\mu \in \mathfrak{M}_{x}^{\inv}(\mathcal{U},M)$ such that $\mu$ is Borel-definable and invariantly supported. Let $r:S_{\mathbf{x}}(\mathcal{U}) \to \prod S_{x}(\mathcal{U})$ via $r(p) = (p|_{x_i})_{i \geq 1}^{\omega}$. Then

\begin{enumerate}
    \item $\{r(p): p \in \supp(\mathbb{P}_{\mu})\} = \{(p_i)_{i \geq 1}^{\omega}: p_i \in \supp(\mu)\} $. 
    \item If $\mu$ is definable then
    \begin{equation*}
        \left\{\bigotimes_{i=1}^{\omega} p_i(x_i): p_i \in \supp(\mu) \right\} \subseteq_{\mathrm{dense}} \supp(\mathbb{P}_{\mu}).
    \end{equation*}
    \item If $\mu = \sum_{i=1}^{n} r_i\delta_{p_i}$ for some $p_{1},\ldots,p_{n} \in S_{x}^{\inv}(\mathcal{U},M)$ and positive real numbers $r_1,\ldots,r_n$ such that $\sum_{i=1}^{n} r_i = 1$, then 
    \begin{equation*}
        \supp(\mathbb{P}_{\mu}) = \left\{\bigotimes_{i=1}^{\omega} q_i(x_i) : q_i \in \{p_1,\ldots,p_n\} \right\}. 
    \end{equation*}
\end{enumerate}
\end{proposition}

\begin{proof} We prove the statements:
\begin{enumerate}
    \item It is clear that $\{r(p): p \in \supp(\mathbb{P}_{\mu})\} \subseteq \{(p_i)_{i \geq 1}^{\omega}: p_i \in \supp(\mu)\}$. It suffices to prove the other direction. Suppose that for each $i \geq 1$, $ p_i \in \supp(\mu)$. We build a type in $\supp(\mathbb{P}_{\mu})$: 
    \begin{enumerate}
        \item Step 1: Choose $q_1(x_1) = p_1(x_1)$. 
        \item Step $n$: Suppose we have constructed a type $q_{n}(x_1,\ldots,x_n)$ such that $q_{n} \in \supp(\mu^{(n)})$ and for each $i \leq n$, $q_{n}|_{x_i} = p_i(x_i)$. By (1) of Proposition \ref{prop:support}, there exists some $q_{n+1} \in \supp(\mu_{x_{n+1}} \otimes \mu^{(n)}_{x_1,\ldots,x_n})$ such that $q_{n+1}|_{x_{n+1}} = p_{n+1}(x_{n+1})$ and $q_{n+1}|_{x_1,\ldots,x_{n}} = q_{n}(x_1,\ldots,x_n)$. 
        \item Consider the type $q = \bigcup_{n =1}^{\omega} q_{n}(x_1,\ldots,x_n)$. We claim that $q$ has the desired properties. 
    \end{enumerate}
    \item By induction and (3) of Proposition \ref{prop:support}, it is straightforward to show that for every $n \geq 1$, $\bigotimes_{i=1}^{n} p_i(x_i) \in \supp(\mu^{(n)})$. It follows quickly that for every sequence $p_1,p_2,p_3,\ldots$ of elements from $\supp(\mu)$, $\bigotimes_{i=1}^{\omega} p_i(x_i) \in \supp(\mathbb{P}_{\mu})$. 
    
    We now prove the density claim. We first show that for every $n \geq 1$,   $\{\bigotimes_{i=1}^{n} p_i(x_i): p_i \in \supp(\mu)\}$ is dense in $\supp(\mu^{(n)})$. Base case is trivial. Suppose the statement holds for $n$. We now show $n+1$. Suppose that $\mu^{(n+1)}(\theta(x_1,\ldots,x_{n+1}))> 0$. It suffices to find some sequence of types $p_1,\ldots,p_{n+1}$ such that $p_i \in \supp(\mu)$ and $\theta(x_1,\ldots,x_{n+1}) \in \bigotimes_{i=1}^{n+1} p_{i}(x_i)$. Since  $\mu^{(n+1)}(\theta(x_1,\ldots,x_{n+1}))> 0$, there exists some $\epsilon > 0$ such that 
        \begin{align*}
        \epsilon &< \mu^{(n+1)}(\theta(x_1,\ldots,x_{n+1})) \\ &= (\mu_{x_{n+1}} \otimes \mu^{(n)}_{x_1,\ldots,x_{n}})(\theta(x_1,\ldots,x_{n+1})) \\ &= \int_{S_{x}(N)} F_{\mu_{x_{n+1}},N}^{\theta} d\mu^{(n)}|_{N}  \\  &= \int_{S_{\bar{x}}(\mathcal{U})} \left(F_{\mu_{x_{n+1}},N}^{\theta}\circ r_{N} \right) d\mu^{(n)},
    \end{align*}
    where $N$ is a small model containing $M$ and all the parameters from $\theta$ and $r_{N}$ is the standard restriction map from $S_{x}(\mathcal{U})$ to $S_{x}(N)$. Let 
    $\Psi \coloneqq \left( F_{\mu_{x_{n+1}},N}^{\theta} \circ r|_{N}\right)|_{\supp(\mu^{(n)})}$. Since the integrand in  equation above is greater than $\epsilon$, we conclude that there exists some $t \in \supp(\mu^{(n)})$ such that $\Psi(t) > \epsilon$. Thus $\Psi^{-1}((\frac{\epsilon}{2},2))$ is non-empty. Since $\mu$ is definable, $F_{\mu_{x_{n+1}},N}^{\theta}$ is continuous and as consequence, so is $\Psi$. Hence $\Psi^{-1}((\frac{\epsilon}{2},2))$ is an open (non-empty) subset of $\supp(\mu^{(n)})$. By our induction hypothesis, there exists some type $s \in \Psi^{-1}((\frac{\epsilon}{2},2))$ such that $s = \bigotimes_{i=1}^{n} p_i(x_i)$. Let $\bar{b} \models s|_{N}$. Then $F_{\mu_{x_{n+1}}}^{\theta}(s|_{N}) = \mu_{x_{n+1}}(\theta(\bar{b},x_{n+1})) > \frac{\epsilon}{2}$.  By \cite[Proposition 2.7]{chernikov2022definable}, there exists some $q \in \supp(\mu)$ such that $\theta(\bar{b},x) \in q$. Then $\theta(x_1,\ldots,x_{n+1}) \in q(x_{n+1}) \otimes s(x_1,\ldots,x_n) = q(x_{n+1}) \otimes  \bigotimes_{i=1}^{n} p_i(x_i)$, which completes the claim. 
    
    To complete the density proof, notice that it suffices to show that if $\mathbb{P}_{\mu}(\theta(x_1,\ldots,x_n)) > 0$ there exists some sequence $p_1,p_2,p_3\ldots$ from $\supp(\mu)$ such that $\bigotimes_{i=1}^{\omega} p_i(x_i) \in [\theta(x_1,\ldots,x_n)]$. Now, if $\mathbb{P}_{\mu}(\theta(x_1,\ldots,x_n)) > 0$, then $\mu^{(n)}(\theta(x_1,\ldots,x_n)) > 0$. By the previous paragraph, there exists $p_1,\ldots,p_n \in \supp(\mu)$ such that $\theta(x_1,\ldots,x_n) \in \bigotimes_{i=1}^{n} p_i(x_i)$. By the first paragraph, the type $\bigotimes_{j=n+1}^{\omega} q_j(x_j) \otimes \bigotimes_{i=1}^{n} p_i(x_i)$ where for any $j \geq n +1$, $q_j$ is an arbitrary element from $\supp(\mu)$ has the desired property. In particular, $\bigotimes_{j=n+1}^{\omega} q_j(x_j) \otimes \bigotimes_{i=1}^{n} p_i(x_i) \in [\theta(x_1,\ldots,x_n)] \cap \supp(\mathbb{P}_{\mu})$.

   \item Notice that 
    \begin{align*}
        \left( \sum_{i=1}^{n} r_i \delta_{p_i} \right)(x_2) \otimes \left( \sum_{i=1}^{n} r_i \delta_{p_i} \right)(x_1) &= \sum_{i=1}^{n}\sum_{j=1}^{n} r_i r_j \left(\delta_{p_i}(x_2) \otimes \delta_{p_j}(x_1)\right) \\ &=\sum_{i=1}^{n}\sum_{j=1}^{n} r_i r_j \delta_{p_i \otimes p_j}(x_2,x_1). \\
    \end{align*}
    By induction, equality holds. \qedhere
\end{enumerate}
\end{proof} 

The next definition ensures that we do not have  certain kinds of pathologies when working with random generic types. The condition below is \emph{in the same vein as} assuming the types one is working with are \emph{not realized}. In practice, most of the measures one actually cares about admit this property. While this assumption will sometimes not be necessary when it is used, it streamlines proofs and cuts down on casework. Most importantly, it guarantees that we do not sample the same point twice.

\begin{definition} Let $\mu \in \mathfrak{M}_{x}(\mathcal{U})$. We say that $\mu$ \emph{has no realized part} if for every $a \in \mathcal{U}$, $\mu(x = a) = 0$. 
\end{definition}

\begin{proposition}\label{prop:concentrate} Suppose that $\mu \in \mathfrak{M}_{x}^{\inv}(\mathcal{U},M)$, $\mu$ is $M$-\weak, and $\mu$ has no realized part. Then for every $p \in \supp(\mathbb{P}_{\mu})$ and $i \neq j \geq 1$, $x_i \neq x_j \in p$. 
\end{proposition}
\begin{proof} 
Fix $j < i$. If $\varphi(x_i,x_j) \coloneqq (x_i = x_j)$, then for any $q \in S_{x}(M)$ and $b \models q$ we have that
\begin{equation*}
    F_{\mu_{x_i},M}^{\varphi}(q) = \mu_{x_i}(x_i =b) = 0,
\end{equation*}
since $\mu$ has no realized part. Therefore
\begin{equation*}
    \mathbb{P}_{\mu}([(x_i = x_j)]) = (\mu_{x_i} \otimes \mu_{x_j})(x_i = x_j) = \int_{S_{x_{j}}(M)} F_{\mu_{x_i},M}^{\varphi} d\mu_{x_j} = \int_{S_{x_j}(M)} 0 d\mu_{x_j}= 0.  
\end{equation*}
Thus if $x_i = x_j \in p$, then $p \not \in \supp(\mathbb{P}_{\mu})$. By contraposition, the statement holds. 
\end{proof}

\subsection{How to compute?} In order to figure out the probability that a complicated event occurs, one must first know how to compute the probabilities of basic events. The lemmas in this section are a guide to such computations. The following fact was first observed in \cite[Lemma 2.14(i), Theorem 3.2((v) $\to$ (ii))]{NIP3}. Neither proof requires NIP, just the appropriate adequate/excellent hypothesis. 

\begin{fact}\label{fact:easy} Let $\mu \in \mathfrak{M}_{x}^{\inv}(\mathcal{U},M)$ and suppose that $\mu$ is $M$-\weak. Then for any formula $\varphi(x_{i_1},\ldots,x_{i_n}) \in \mathcal{L}_{\mathbf{x}}(\mathcal{U})$ such that $i_1 < \ldots < i_n$, 
\begin{equation*}
    \mathbb{P}_{\mu}\left( [\varphi(x_{i_1},\ldots,x_{i_n})] \right) =  \mu^{(n)}(\varphi(x_1,\ldots,x_{n})).
\end{equation*}    
If $\mu$ is $M$-\strong\ then for any formula $\varphi(x_{j_1},\ldots,x_{j_n}) \in \mathcal{L}_{\mathbf{x}}(\mathcal{U})$ such that the indices $j_1,\ldots,j_n$ are pairwise distinct, 
\begin{equation*}
    \mathbb{P}_{\mu}([\varphi(x_{j_1},\ldots,x_{j_n})]) = \mu^{(n)}(\varphi(x_1,\ldots,x_n)). 
\end{equation*}
\end{fact}

The following theorem shows that formulas with disjoint variables are probabilistically independent. 

\begin{theorem}\label{lemma:product} Fix $\mu \in \mathfrak{M}_{x}^{\inv}(\mathcal{U},M)$ such that $\mu$ is $M$-\weak. Let $\{\varphi_{l}(\bar{z}_{l})\}_{l=1}^{n}$ be a collection of $\mathcal{L}_{\mathbf{x}}(\mathcal{U})$-formulas such that: 
\begin{enumerate}
    \item For each $l \leq n$, $\bar{z}_{l} = (x_{j_1^{l}},\ldots,x_{j_{k_{l}}^{l}})$. 
    \item For each $l_1 < l_2 \leq n$ we have that $\{j_{1}^{l_{1}},\ldots,j_{k_{l_1}}^{l_1}\} \cap  \{j_{1}^{l_{2}},\ldots,j_{k_{l_2}}^{l_2} \} = \emptyset$. 
\end{enumerate}
Then 
\begin{equation*}
\mathbb{P}_{\mu} \left( \left[\bigwedge_{l=1}^{n} \varphi_l(\bar{z}_l) \right] \right) = \prod_{l = 1}^{n} \mathbb{P}_{\mu}\left( \left[ \varphi_l(\bar{z}_l) \right]\right) =  \prod_{l=1}^{n} \mu^{(|\bar{z}_l|)}(\varphi_{l}(x_1,\ldots,x_{k_l})).
\end{equation*}
\end{theorem}
\begin{proof}
This follows directly by induction and the fact that the Morley product is a separated amalgam (see Definition \ref{cheat}). We give an explicit proof for two formulas with two variables each, which, with proper bookkeeping, can be extended to the general case in the lemma: Consider the $\mathcal{L}_{\mathbf{x}}(\mathcal{U})$-formulas $\varphi(x_1,x_3)$ and $\psi(x_2,x_4)$. Let $M_1$ be a small model containing $M$ and all the parameters from our formulas. Then by Fact \ref{fact:easy},
\begin{align*}
    \mathbb{P}_{\mu} \left( \left[ \varphi(x_1,x_3) \wedge \psi(x_2,x_4) \right]\right) &= (\mu_{x_4} \otimes \ldots \otimes \mu_{x_1})(\varphi(x_1,x_3) \wedge \psi(x_2,x_4)) \\
    &= \int_{S_{x_1}(M_1)} F_{\mu_{x_4} \otimes \mu_{x_3} \otimes \mu_{x_2}}^{\varphi \wedge \psi} d\mu_{x_1}. 
\end{align*}
Now for any $q \in S_{x_1}(M_1)$, if $a \models q$ then we have that 
\begin{align*}
    F_{\mu_{x_4} \otimes \mu_{x_3} \otimes \mu_{x_2}}^{\varphi \wedge \psi}(q) &= (\mu_{x_4} \otimes \mu_{x_3} \otimes \mu_{x_2})(\varphi(a,x_3) \wedge \psi(x_2,x_4))\\
    &= \int_{S_{x_2}(M_2)} F_{\mu_{x_4} \otimes \mu_{x_3}}^{(\varphi \wedge \psi)_a} d\mu_{x_2},
\end{align*}
where $M_2$ is a small model containing $M_1a$ and $(\varphi \wedge \psi)_{a}$ is the formula $\varphi(a,x_3) \wedge \psi(x_2,x_4)$. Now we note that for any $r \in S_{x_2}(M_2)$, if $b \models r$ then
\begin{align*}
 F_{\mu_{x_4} \otimes \mu_{x_3}}^{(\varphi \wedge \psi)_a}(r) &= (\mu_{x_4} \otimes \mu_{x_3})(\varphi(a,x_3) \wedge \psi(b,x_4)) \\
 &\overset{(*)}{=} \mu_{x_3}(\varphi(a,x_3)) \mu_{x_4}(\psi(b,x_4)) \\
 &= \mu_{x_3}(\varphi(a,x_3)) \cdot F_{\mu_{x_4}}^{\psi}(r),
\end{align*}
where $(*)$ holds since the Morley product is a separated amalgam (again, see Definition \ref{cheat}). And so, 
\begin{align*}
    \int_{S_{x_2}(M_2)} F_{\mu_{x_4} \otimes \mu_{x_3}}^{(\varphi \wedge \psi)_a} d\mu_2 &= \int_{S_{x_2}(M_2)} \mu_{x_3}(\varphi(a,x_3)) \cdot F_{\mu_{x_4}}^{\psi}d\mu_2 \\
    &= \mu_{x_3}(\varphi(a,x_3))\int_{S_{x_2}(M_2)}  F_{\mu_{x_4}}^{\psi} d\mu_2 \\
    &= \mu_{x_3}(\varphi(a,x_3)) \cdot (\mu_{x_4} \otimes \mu_{x_2})(\psi(x_2,x_4)) \\
    &= F_{\mu_{x_3}}^{\varphi}(q) \cdot (\mu_{x_4} \otimes \mu_{x_2})(\psi(x_2,x_4)).
\end{align*}
Therefore, 
\begin{align*}
 \mathbb{P}_{\mu} \left( [\varphi(x_1,x_3) \wedge \psi(x_2,x_4)]\right) &= \int_{S_{x_1}(M_1)} F_{\mu_{x_4} \otimes \mu_{x_3} \otimes \mu_{x_2}}^{\varphi \wedge \psi} d\mu_{x_1} \\
 &= \int_{S_{x_1}(M_1)} F_{\mu_3}^{\varphi} \cdot (\mu_{x_4} \otimes \mu_{x_2})(\psi(x_2,x_4)) d\mu_{x_1} \\ 
 &= (\mu_{x_4} \otimes \mu_{x_2})(\psi(x_2,x_4)) \cdot \int_{S_{x_1}(M_1)} F_{\mu_3}^{\varphi}  d\mu_{x_1} \\ 
 &= (\mu_{x_4} \otimes \mu_{x_2})(\psi(x_2,x_4)) \cdot (\mu_{x_3} \otimes \mu_{x_1})(\varphi(x_1,x_3)) \\
 &= \mathbb{P}_{\mu}([\psi(x_2,x_4)]) \cdot \mathbb{P}_{\mu}([\varphi(x_1,x_3)]). 
\end{align*}
We also remark that by renaming variables (Fact \ref{fact:easy}) we have the following equality,
\begin{equation*}
    (\mu_{x_4} \otimes \mu_{x_2})(\psi(x_2,x_4)) \cdot (\mu_{x_3} \otimes \mu_{x_1})(\varphi(x_1,x_3)) = \mu^{(2)}(\psi(x_1,x_2)) \cdot \mu^{(2)}(\varphi(x_1,x_2)).
\end{equation*}
And so the proof is complete. 
\end{proof}

The next lemma shows us how to compute when there are some basic dependencies between the formulas in our family of definable sets (specifically a block of shared variables followed by disjoint families of variables as in \cref{lemma:product}). This computation will become quite important in the next section. 

\begin{lemma}\label{label:conditional} Let $\mu \in \mathfrak{M}_{x}^{\inv}(\mathcal{U},M)$. Suppose that $\mu$ is $M$-\weak. Let $\{\varphi_{l}(\bar{x},\bar{z}_l)\}_{l = 1}^{n}$ be a collection of $\mathcal{L}_{\mathbf{x}}(\mathcal{U})$-formulas such that 
\begin{enumerate}
    \item $\bar{x} = (x_{i_1},\ldots,x_{i_m})$ and for each $l \leq n$, $\bar{z}_l = (x_{j^{l}_1},\ldots,x_{j^{l}_{k_l}})$. 
    \item $i_1 < \ldots < i_m$ and for each $l \leq n$, $j_1^{l} < \ldots < j_{k_l}^{l}$. 
    \item For each $l \leq n$, $\max\{i_1,\ldots,i_m\} < \min\{ j^{l}_1,\ldots,j^{l}_{k_l}\}$. 
    \item For each $l_1, l_2 \leq n$, if $l_1 < l_2$ then $\{ j_1^{l_1},\ldots,j_{k_l}^{l_1}\} \cap  \{ j_1^{l_2},\ldots,j_{k_l}^{l_2}\} = \emptyset$.
\end{enumerate}
Then 
\begin{equation*}
    \mathlarger{\mathbb{P}_{\mu} \left(\bigcap_{l=1}^{n} [\varphi_l(\bar{x},\bar{z}_l)] \right) = \int_{S_{\bar{x}}(M')} \prod_{l=1}^{n}F_{\mu_{\bar{z}_l}}^{\varphi_l}\,d\mu_{\bar{x}}}
\end{equation*}
where
\begin{enumerate}
    \item $M \subseteq M'$ and $M'$ contains all the parameters from $\{\varphi_{l}(\bar{x},\bar{z}_{l})\}_{l=1}^{n}$. 
    \item $\mu_{\bar{x}} = \mu_{x_{i_{m}}} \otimes \ldots \otimes  \mu_{x_{i_1}}$. 
    \item For each $l \leq n$, $\mu_{\bar{z}_{l}} = \mu_{x_{j^{l}_{k_{l}}}} \otimes \ldots \otimes \mu_{x_{j_{1}^{l}}}$. 
\end{enumerate}
\end{lemma}

\begin{proof}  We let $\psi_{l}(\bar{x},\bar{z}) \coloneqq \bigwedge_{l=1}^{n} \varphi_l(\bar{x},\bar{z}_l)$. Let $\bigcup_{l=1}^{n}\bar{z}_{l} = \{x_{t_{1}} < \ldots < x_{t_{q}}\}$. Now consider the following computation:
\begin{align*}
    \mathbb{P}_{\mu} \left( \bigcap_{l=1}^{n} [\varphi_l(\bar{x},\bar{z}_l)] \right) &= \mathbb{P}_{\mu} \left( \left[ \bigwedge_{l=1}^{n} \varphi_l(\bar{x},\bar{z}_l) \right] \right) \\
    &=  \left( \mu_{x_{t_{q}}} \otimes \left( \ldots \otimes \left( \mu_{x_{t_1}} \otimes \mu_{\bar{x}} \right) \ldots \right) \right) \left( \left[ \mathlarger{\bigwedge_{l=1}^{n}} \varphi_l(\bar{x},\bar{z}_{l}) \right] \right)\\
    &\overset{(*)}{=}  \left(\left(  \mu_{x_{t_q}} \otimes \left(  \ldots \otimes \mu_{x_{t_1}} \right) \ldots \right) \otimes  \mu_{\bar{x}} \right) \left( \left[ \mathlarger{\bigwedge_{l=1}^{n}} \varphi_l(\bar{x},\bar{z}_{l}) \right] \right)\\
    &= \int_{S_{\bar{x}}(M')} F_{\mu_{x_{t_q}} \otimes \ldots \otimes \mu_{x_{t_1}} }^{\psi} d \mu_{\bar{x}}\\
    &\overset{(**)}{=} \int_{S_{\bar{x}}(M')} \mathlarger{\prod_{l=1}^{n}} F_{\mu_{\bar{z}_l}}^{\varphi_l} d\mu_{\bar{x}}.
\end{align*}
We remark that equation $(*)$ follows from self-associativity. We now justify equation $(**)$. For any $q \in S_{\bar{x}}(M')$, if $\bar{a} \models q$ then we have the following:
\begin{align*}
    F_{\mu_{x_{t_{q}}} \otimes \ldots \otimes \mu_{x_{t_1}} }^{\psi}(q) &= \left(\mu_{x_{t_q}} \otimes \ldots \otimes \mu_{x_{t_1}} \right) \left( \bigwedge_{l=1}^{n} \varphi_l(\bar{a},\bar{z}_l) \right) \\
    &\overset{(\dagger)}{=}\mathlarger{\prod_{l=0}^{n} \mu_{\bar{z}_{l}}(\varphi_l(\bar{a},\bar{z}_l))} \\
    &=\mathlarger{\prod_{l=1}^{n} F_{\mu_{\bar{z}_l}}^{\varphi_l}(q)}.
\end{align*}
Equation $(\dagger)$ follows by \cref{lemma:product}. 
\end{proof}

The following is essentially an application of Borel-Cantelli. For clarity, we provide a proof.

\begin{proposition}\label{prop:inf-often}  Let $\mu \in \mathfrak{M}_{x}^{\inv}(\mathcal{U},M)$. Suppose that $\mu$ is $M$-\weak. For any $\mathcal{L}_{\mathbf{x}}(\mathcal{U})$-formula $\varphi(x_1,\ldots,x_n)$, we can consider the event
\begin{equation*}
    A_{\varphi} \coloneqq \{p \in S_{\mathbf{x}}(\mathcal{U}): \exists^{\infty} (i_1,\ldots,i_n) \text{ such that } \varphi(x_{i_1},\ldots,x_{i_n}) \in p.\}
\end{equation*}
If $\mu^{(n)}(\varphi(x_1,\ldots,x_{n})) > 0$, then $\mathbb{P}_{\mu}(A_{\varphi}) =1$. 
\end{proposition}

\begin{proof} We first notice that 

\begin{equation*}
 \bigcap_{ t = 1}^{\omega} \left(  \bigcup_{ t < i_1 <\ldots < i_n } [\varphi(x_{i_1},\ldots,x_{i_n})] \right) \subseteq A_{\varphi}. 
\end{equation*}
Hence it suffices to prove the term on the left is measure 1. Fix $t \geq 1$. Notice that
\begin{align*}
    \mathbb{P}_{\mu} \left( \bigcup_{ t < i_1 <\ldots < i_n } [\varphi(x_{i_1},\ldots,x_{i_n})]  \right) &\geq \mathbb{P}_{\mu}\left( \bigcup_{l > t}^{\omega} [\varphi(x_{l \cdot n},\ldots,x_{l\cdot n+n})]\right) \\
    &= \lim_{\substack{k \to \infty\\ k > t}} \mathbb{P}_{\mu}\left( \left[ \bigvee_{l >t}^{k} \varphi(x_{l \cdot n},\ldots,x_{l\cdot n+n}) \right]\right) \\
    &= 1- \lim_{\substack{k \to \infty\\ k > t}} \mathbb{P}_{\mu}\left(\left[\bigwedge_{l=t}^{k} \neg \varphi(x_{ln},\ldots,x_{ln+n}) \right]\right)\\
    &\overset{(*)}{=} 1- \lim_{\substack{k \to \infty\\ k > t}} \left( 1 -\mu^{(n)}( \varphi(x_1,\ldots,x_n)) \right)^{k -t}\\
    &= 1 
\end{align*}
Equation $(*)$ follows from \cref{lemma:product} and Fact \ref{fact:easy}, and since the intersection of countably many sets of full measure has full measure, the statement holds. 
\end{proof}

\section{Induced substructures} 

Throughout this section, $\mathcal{L}$ is a countable relational language. We address the following question:

\begin{question}\label{question:main}
Given an $M$-\weak\ measure $\mu \in \mathfrak{M}_{x}^{\inv}(\mathcal{U},M)$, is there an $\mathcal{L}$-structure $N$ such that the induced structure on almost all random generic types is isomorphic to $N$? 
\end{question}

First, we show how to turn non-redundant types in infinitely many variables into $\mathcal{L}$-structures.

\begin{definition} Consider $\mathcal{U}'$ where $\mathcal{U} \prec \mathcal{U}'$ and $\mathcal{U}'$ is $|\mathcal{U}|^{+}$-saturated. Given a type $p \in S^{+}_{\mathbf{x}}(\mathcal{U})$ and a tuple $\mathbf{a} = (a_1,a_2,\ldots) \in (\mathcal{U}')^{\omega}$ such that $\mathbf{a} \models p$, one can consider the $\mathcal{L}$-structure $\mathbf{a}_{\mathcal{L}}$ which is simply the induced structure from $\mathcal{U}'$ onto the set $\{a_i: i \geq 1\}$. We remark that if $\mathbf{a}, \mathbf{b} \models p$ then $\mathbf{a}_{\mathcal{L}} \cong \mathbf{b}_{\mathcal{L}}$. In particular, the map $f(a_i) = b_i$ is an isomorphism between the induced structures.  Hence we let $p_{\mathcal{L}}$ be the isomorphism type of any/all the realizations of $p$. One can also think about this process as constructing a structure on the set $\{x_i: i \geq 1\}$ where the points in the model are the variables, but this causes a lot of notational confusion and so we will avoid writing this. 
\end{definition} 

Formalizing Question \ref{question:main}: given an $M$-\weak\ measure $\mu$, does there exist an $\mathcal{L}$-structure $N$ such that the event
\begin{equation*}
    B_{N} \coloneqq \{p \in S_{\mathbf{x}}(\mathcal{U}): p_{\mathcal{L}} \cong N\},
\end{equation*}
has $\mathbb{P}_{\mu}$-measure 1? There are several \emph{relatively easy contexts} in which the answer to the above question is yes. We will briefly discuss these later in this section. There are also quite a few examples where the answer is no. Our main affirmative result occurs in the following setting: If $\mathcal{L}$ is a finite relational language and $\mu$ satisfies a particular extension axiom, then there exists an $\mathcal{L}$-structure $N$ with the desired properties. Moreover, $\Th_{\mathcal{L}}(N)$ is countably categorical.

To address this, we now make explicit the connection between generic sampling and the space of labeled $\mathcal{L}$-structures. We use this transfer map to show that the events of the form $B_N$ are $\mathbb{P}_{\mu}$-measurable. 

\begin{definition}\label{def:G_map} We define a map $g: S_{\mathbf{x}}^{+}(\mathcal{U}) \to \str_{\mathcal{L}}$ where $g(p) = M_{p}$ is defined by
\begin{equation*}
    M_p \models R(i_1,\ldots,i_n) \Longleftrightarrow R(x_{i_1},\ldots,x_{i_n}) \in p. 
\end{equation*}
We remark that $g$ is continuous. 
\end{definition}

\begin{remark}\label{remark:connection} Suppose that $\mu \in \mathfrak{M}_{x}^{\inv}(\mathcal{U},M)$ is $M$-adequate and has no realized part. Then by Proposition \ref{prop:concentrate} we have that $\supp(\mathbb{P}_{\mu}) \subseteq S_{\mathbf{x}}^{+}(\mathcal{U})$. Thus by Definition \ref{def:G_map}, we can construct a measure on the space of labeled $\mathcal{L}$-structures by considering the push-forward of $\mathbb{P}_{\mu}$ by $g$. 
\end{remark}

The iterated Morley product and its associated measure on the space of labeled $\mathcal{L}$-structures agree on quantifier-free formulas. 

\begin{proposition}\label{prop:qf}  Suppose that $\mu \in \mathfrak{M}_{x}^{\inv}(\mathcal{U},M)$ is $M$-adequate and has no realized part. Then for any quantifier-free $\mathcal{L}$-formula $\theta(x_{i_1},\ldots,x_{i_n})$ and distinct integers $i_1,\ldots,i_n$, 
\begin{equation*}
    g_*(\mathbb{P}_{\mu})(\llbracket\theta(i_1,\ldots,i_n)\rrbracket) = \mathbb{P}_{\mu}(\theta(x_{i_1},\ldots,x_{i_n})). 
\end{equation*}
\end{proposition}

\begin{proof} Follows directly from the definition of $g$. 
\end{proof}

The above proposition cannot be extended to other kinds of formulas. 

\begin{warning}  Consider $M = (\mathbb{R},<) \prec \mathcal{U}$ and $p = \{x > a: a \in \mathcal{U}\}$. Then $p$ is $\emptyset$-definable and thus definable over $M$. This implies that $p$ is $M$-adequate. Now $\mathbb{P}_{\delta_{p}}$ is just the Dirac measure corresponding to the type of a Morley sequence in $p$ over $\mathcal{U}$. Thus $g_{*}(\mathbb{P}_{\delta_p})$ is the Dirac measure concentrating on the labeled $\mathcal{L}$-structure $M = (\mathbb{N}_{>0},<)$, where the labels align with the appropriate numbers. Consider the formula $\psi(x_1) := \exists y (y < x_1)$. Then 
\begin{equation*}
    \mathbb{P}_{\delta_{p}}(\psi(x_1)) = 1 \neq 0 = g_{*}(\mathbb{P}_{\delta_{p}})(\llbracket \psi(1) \rrbracket) =\delta_{M}(\llbracket \psi(1) \rrbracket). 
\end{equation*}
\end{warning}

Before continuing any further, we remark that for any $\mathcal{L}$-structure $N$, the event $B_{N}$ is $\mathbb{P}_{\mu}$-measurable whenever $\mu$ has no realized part.

\begin{proposition}\label{prop:borel} Let $\mu \in \mathfrak{M}_{x}^{\inv}(\mathcal{U},M)$, and suppose that $\mu$ is $M$-adequate and has no realized part. Then for any $\mathcal{L}$-structure $N$, the event $B_{N} \subseteq S_{\mathbf{x}}(\mathcal{U})$ is $\mathbb{P}_{\mu}$-measurable (with respect to the completion).
\end{proposition}

\begin{proof} Since $\mu$ has no realized part, Proposition \ref{prop:concentrate} implies that $\supp(\mathbb{P}_{\mu}) \subseteq S_{\mathbf{x}}^{+}(\mathcal{U})$.
We remark that $p \in B_{N} \cap S_{\mathbf{x}}^{+}(\mathcal{U})$ if and only if $M_p \cong N$. Hence, 
\begin{equation*}B_{N} \cap S_{\mathbf{x}}^{+}(\mathcal{U}) = g^{-1}(\{ M \in \str_{\mathcal{L}} : M \cong N \}). 
\end{equation*}
A classical result of Scott implies that $\{M \in \str_{\mathcal{L}} : M \cong N \}$ is Borel (e.g., see \cite[II.16.6]{Kechris}). Since $g$ is continuous, the statement holds.
\end{proof}

\subsection{Some relatively easy contexts} Before working toward our main result, we briefly consider some relatively easy contexts where the answer to Question \ref{question:main} is yes. While the proofs in this subsection are straightforward, the statements provide useful sanity checks for the kinds of theorems one might expect to prove. 

\begin{proposition} Suppose that $p \in S^{\inv}_{x}(\mathcal{U},M)$. Consider the measure $\mu = \delta_{p} \in \mathfrak{M}_{x}^{\inv}(\mathcal{U},M)$. Then there exists a structure $N$ such that $\mathbb{P}_{\mu}(B_{N}) = 1$. 
\end{proposition}

\begin{proof} Notice that $|\supp(\mathbb{P}_{\mu})| = |\{p^{(\omega)}\}|= 1$. So $N$ is precisely the isomorphism type of $(p^{(\omega)})_{\mathcal{L}}$. 
\end{proof}

A priori, the measure $\mu$ in the previous proposition might not be Borel-definable. This is not a problem because the Morley product can be extended to a more general context where one asks that certain measurability conditions hold. These conditions always hold when the measure is a type. See \cite{CGH} for a discussion.

The next result shows that for any countable $\mathcal{L}$-structure $N$, there exists a measure $\mu_{N}$ such that $\mathbb{P}_{\mu_{N}}(B_{N}) = 1$. This is done by simply choosing a sum of realized types which concentrates on $N$. However, we note that the measures in the next example \emph{have realized parts}. 

\begin{proposition} Let $N$ be any countable subset of $M$ and consider the measure $\mu = \sum_{a \in N} r_{a} \delta_{\tp(a/\mathcal{U})}$ where $\sum_{a \in N} r_{a} = 1$. Then $\mu \in \mathfrak{M}_{x}^{\inv}(\mathcal{U},M)$ and $\mu$ is smooth (and therefore definable). Moreover $\mathbb{P}_{\mu}(B_{N}) = 1$.
\end{proposition}

\begin{proof} The proof is essentially Borel-Cantelli and is left to the reader as an exercise.
\end{proof}

We remark that sometimes one can prove some easy examples by hand, especially when they are constructed from a finite sum of types. In these cases, one can think of random generic types as paths through a finitely branching tree. At step $n+1$, we choose one of the finitely many types in the support of our measure and take the Morley product with the type we constructed at step $n$ (see Proposition \ref{prop:bigsupport}). However, not all finite sums of types yield a positive answer to Question \ref{question:main} (see Example \ref{example:bad}) and so a case-by-case analysis is necessary. We leave the following example as an exercise as well. Hint: Borel-Cantelli.   

\begin{proposition} Consider $(\mathbb{R},<) \prec (\mathcal{U},<)$. Consider the types $p_1,p_2, p_3,p_4\in S_{x}^{\inv}(\mathcal{U},\mathbb{R})$ such that 
\begin{enumerate}
    \item $p_1 \supseteq \{x > a : a \in \mathcal{U}\}$.
    \item $p_2 \supseteq \{x < a : a \in \mathcal{U}\}$. 
    \item $p_3 \supseteq \{x > a: a \in \mathbb{R}\} \cup \{x < a : \forall b \in \mathbb{R}, a > b\}$. 
    \item $p_4 \supseteq \{x < a: a \in \mathbb{R}\} \cup \{x > a : \forall b \in \mathbb{R}, a < b\}$. 
\end{enumerate}
Let $\mu_1 = \frac{1}{2}\delta_{p_1} + \frac{1}{2}\delta_{p_2}$ and $\mu_2 = \frac{1}{2}\delta_{p_3} + \frac{1}{2}\delta_{p_4}$. Let $N_1 = (\mathbb{Z}, <)$ and $N_2 = (\mathbb{N} + \mathbb{N}^{*},<)$. Then $\mathbb{P}_{\mu_1}(B_{N_1}) = 1$ while  $\mathbb{P}_{\mu_2}(B_{N_2}) = 1$.
\end{proposition}

Our final example of this section is more general. We show that if a measure concentrates on finitely many types which pairwise commute (including with themselves), then we obtain a positive answer to Question \ref{question:main}. This happens, for example, if our measure $\mu$ is the sum of finitely many generically stable types. 

\begin{theorem} Suppose  $\mu = \sum_{i=1}^{n} r_i q_i$ where each type in $\{q_1,\ldots,q_n\}$ is $M$-invariant, $\sum_{i=1}^{n} r_i = 1$, and for each pair $1 \leq i, j \leq n$, we have that $q_i(x) \otimes q_j(y) = q_j(y) \otimes q_i(x)$. If $\mu$ has no realized part, then there exists an $\mathcal{L}$-structure $N$ such that $\mathbb{P}_{\mu}(B_{N}) = 1$. 
\end{theorem}

\begin{proof} By Proposition \ref{prop:bigsupport}, every element of $\supp(\mathbb{P}_{\mu})$ is an iterated Morley product of elements in $\{q_1,\ldots,q_n\}$. For each $k \leq n$ we let $W_{k} \coloneqq \{ p \in S_{\mathbf{x}}(\mathcal{U}) : \exists^{\infty} j, p|_{x_j} = q_k(x_j)\}$. By Borel-Cantelli, we claim that $\mathbb{P}_{\mu}(W_{k}) = 1$. Let $W = \bigcap_{k=1}^{n} W_k$. Now $\mathbb{P}_{\mu}(W) = 1$. We claim that if $p, s \in W \cap \supp(\mathbb{P}_{\mu})$,
then $p|_{\mathcal{L}} \cong s|_{\mathcal{L}}$. Indeed, suppose that $(a_1,a_2,\ldots) \models p|_{\mathcal{L}}$ and $(b_1,b_2,\ldots) \models s|_{\mathcal{L}}$. Notice that since $\mu$ has no realized part, we have that $a_i = a_j$ if and only if $i = j$ by Proposition \ref{prop:concentrate}. Similarly for the $b_i$'s. We prove they are isomorphism using a back-and-forth method.  
\begin{enumerate}
    \item For $a_1$, choose the first $b_t$ such that $\tp(a_1/\mathcal{U}) = \tp(b_t/\mathcal{U})$. 
    \item Now suppose we have a map $f:\{a_{i_1},\ldots,a_{i_m}\} \to \{b_{j_1},\ldots,b_{j_m}\}$. Let $l$ be the smallest index which has not yet appeared in $\{i_1,\ldots,i_m\}$. By commutativity, we have that $\tp(a_{i_1},\ldots,a_{i_m},a_{l}/\mathcal{U}) = \tp(a_{i_1},\ldots,a_{i_m}/\mathcal{U}) \otimes \tp(a_{l}/\mathcal{U})$. By our induction hypothesis, we have that $\tp(a_{i_1},\ldots,a_{i_m}/\mathcal{U}) = \tp(b_{j_1},\ldots,b_{j_m}/\mathcal{U})$. Now choose the smallest index $t$ such that $\tp(b_t/\mathcal{U}) = \tp(a_{l}/\mathcal{U})$. Then, 
    \begin{align*}
    \tp(a_{i_1},\ldots,a_{i_m},a_{l}/\mathcal{U}) &= \tp(a_{i_1},\ldots,a_{i_m}/\mathcal{U}) \otimes \tp(a_{l}/\mathcal{U}) \\
        &= \tp(b_{j_1},\ldots,b_{j_m}/\mathcal{U}) \otimes \tp(b_{t}/\mathcal{U}) \\
        &= \tp(b_{j_1},\ldots,b_{j_m},b_{t}/\mathcal{U}). \\
    \end{align*}
    Let $f(a_{l}) = b_t$.
    \item A similar argument allows us to \emph{go back}. 
\end{enumerate}
We claim that $f$ is an isomorphism. Thus we can just let $N$ be the isomorphism type of any type $p \in \supp(\mathbb{P}_{\mu}) \cap W$. 
\end{proof}

\subsection{The measure extension axioms} Moving on, we now aim to prove the main theorem of this section. We prove that if a measure $\mu$ witnesses the \emph{\strong\ measure extension axiom}, then there exists some $\mathcal{L}$-structure $N$ such that $\mathbb{P}_{\mu}(B_{N}) = 1$. 

An intuitive explanation behind the extension axioms is the following: Suppose, with positive probability, one can generically sample a certain $(n+1)$-quantifier-free type, say $q(x_1,\ldots,x_{n+1})$. We let $q'(x_1,\ldots,x_{n})$ be the truncation of the quantifier-free type $q$ to the first $n$ variables. An adequate measure $\mu$ witnesses the \emph{adequate extension axiom} if whenever we generically sample $n$ points which witness $q'$, then with positive probability, we can sample a new point such that the entire tuple witnesses $q$. If $\mu$ is also excellent, we say that it satisfies the \emph{excellent extension axiom}.   

\begin{definition} Suppose that $\mathcal{L}$ is a finite relational language. If $\bar{a}\coloneqq a_1,\ldots,a_n \in \mathcal{U}$,
we let $\diam_{\bar{a}}(x_1,\ldots,x_n)$ be the quantifier-free diagram of the tuple $(a_1,\ldots,a_n)$. If $q$ is a complete quantifier-free type in variables $x_1,\ldots,x_n$ over $\emptyset$, then we let $\diam_{q}(x_1,\ldots,x_n) = \diam_{\bar{a}}(x_1,\ldots,x_n)$ where $\bar{a} \models q$. 
\end{definition}

\begin{definition} Fix $\mu \in \mathfrak{M}_{x}^{\inv}(\mathcal{U},M)$. We say that $\mu$ satisfies the \emph{\weak\ measure extension axiom} if $\mu$ is $M$-\weak\ and for any $a_1,\ldots,a_{n+1} \in \mathcal{U}$, if $\bar{a} = a_1,\ldots,a_n$ and 
\begin{equation*}
\mu_{x_{n+1}} \otimes \mu^{(n)}_{x_1,\ldots,x_n}(\diam_{\bar{a},a_{n+1}}(\bar{x},x_{n+1})) > 0,
\end{equation*}
then 
\begin{equation*}
    \mu^{(n)}_{x_1,\ldots,x_n}|_{M} \left( [\diam_{\bar{a}}(x_1,\ldots,x_n)] \cap \left(F_{\mu_{x_{n+1}},M}^{\varphi}\right)^{-1}(\{1\}) \right) = 0. 
\end{equation*}
where $\varphi(x_{n+1};x_1,\ldots,x_n) \coloneqq \neg \diam_{\bar{a},a_{n+1}}(x_1,\ldots,x_n,x_{n+1})$. Since the variables are a little tricky, we recall that for any $q \in S_{x_1,\ldots,x_n}(M)$, we have that 
\begin{equation*}
    F_{\mu_{x_{n+1}},M}^{\varphi}(q) = \mu_{x_{n+1}}(\neg \diam_{\bar{a}}(c_1,\ldots,c_n,x_{n+1})).
\end{equation*}
where $(c_1,\ldots,c_n) \models q$. 

In other words: It says that if some quantifier-free diagram of size $n+1$ has positive measure, then the measure of types extending the $n$-type that fail to extend to the full $n+1$ diagram is zero.

We say that $\mu$ satisfies the \emph{\strong\ measure extension axiom} if $\mu$ is $M$-\strong\ and $\mu$ satisfies the \weak\ measure extension axiom. 
\end{definition}

\begin{example} We recall our two main examples from Example \ref{example:main}. We claim that both witness the \strong\ measure extension axiom. We only sketch these arguments since these examples will be worked with rigorously in the next section.
\begin{enumerate} 
    \item First we consider Example \ref{example:main}(1) with measure $\mu_{L}$. We claim that $\mu_{L}$ satisfies the \strong\ measure extension axiom. Indeed, notice that a quantifier-free type in variables $x_1,\ldots,x_n$ has positive measure if and only if there is no equality between the variables, i.e.\ we have that 
    \begin{align*}
        \mu_{L}^{(n+1)}&(\diam_{a_1,\ldots,a_{n+1}}(x_1,\ldots,x_{n+1})) > 0 \\  &\mathrel{\Longleftrightarrow} {\models} \bigvee_{\sigma \in \Sym(n+1)} a_{\sigma(1)} < \ldots < a_{\sigma(n+1)}. 
    \end{align*}
    The cases are more or less similar, so we only consider the following quantifier-free type over the empty set, $q(x_1,\ldots,x_n,x_{n+1}) \coloneqq x_1 < \ldots < x_{n} < x_{n+1}$. Notice that for any $s \in [\diam_{q}(x_1,\ldots,x_n)] \cap S_{\bar{x}}(M)$ and $(c_1,\ldots,c_n) \models s$, 
\begin{align*}
    &F_{\mu_{x_{n+1}}}^{\varphi}(q) = 1 \\
    &\Longleftrightarrow \mu_{x_{n+1}}(\neg\diam_{s}(c_1,\ldots,c_n,x_{n+1})) = 1 \\
    &\Longleftrightarrow \mu_{x_{n+1}}(\diam_{s}(c_1,\ldots,c_n,x_{n+1})) = 0\\
    &\Longleftrightarrow \mu_{x_{n+1}}( \diam_{s|_{x_n,x_{n+1}}}(c_n,x_{n+1})) =0 \\
    &\Longleftrightarrow \mu_{x_{n+1}}(c_n < x_{n+1}) = 0 \\
    &\Longleftrightarrow \st(c_n) \geq 1. 
\end{align*}
where $\st(c_n)$ is the standard part of $c_n$. Therefore 
\begin{align*}
    [\diam_{a_1 < \ldots < a_n}(x_1,\ldots,x_n)] \cap& \left( F_{\mu_{n+1},M}^{\varphi} \right)^{-1}(\{1\}) \\ &= \{p \in S_{\bar{x}}(M) : \st(x_n) \geq 1\}.
\end{align*}
Now, 
\begin{equation*}
    \{p \in S_{\bar{x}}(M) : \st(x_n) \geq 1\} \cap \supp(\mu^{(n)}|_{M}) \neq \emptyset, 
\end{equation*}
However, this set has $\mu^{(n)}|_{M}$-measure 0.
 \item Now we consider Example \ref{example:main}(2). It is relatively straightforward to see that $\mu_{t}$ satisfies the \strong\ extension axiom. In this case, any quantifier-free type (of positive measure) in $n$ variables can always be extended to any quantifier-free type (of positive measure) in $n+1$ variables as long as the restriction of the larger type is the smaller one. In this example, the term \emph{almost all} could have been replaced with \emph{all} in the extension axiom.
\end{enumerate}
\end{example}

We will now prove a general computational lemma. We will apply it to measures which satisfy the extension axioms. We will need the following fact. 

\begin{fact}\label{fact:converge} Let $X$ be a topological space, $\mu$ be a Borel probability measure on $X$, and $f:X \to [0,1]$ a Borel function. If $\mu(\{x \in X: f(x) = 1\}) =0$, then 
\begin{equation*}
    \lim_{k \to \infty} \int_{X} f^{k} d\mu = 0. 
\end{equation*}
\end{fact}

\begin{proof} Notice the sequence $(g_n)_{n \geq 1}$ where $g_n = f^{n}$ is a uniformly bounded sequence of functions which converges pointwise to $\mathbf{1}_{f(x)=1}$. By the dominated convergence theorem, 
\begin{equation*}
    \lim_{n \to \infty} \int_{X} g_{n} d\mu = \int_{X} \mathbf{1}_{f(x) = 1}d\mu = \mu(\{x \in X: f(x) = 1\}) = 0. \qedhere
\end{equation*}
\end{proof}

\begin{lemma}\label{lemma:ext} Suppose $\mu \in \mathfrak{M}^{\inv}_{x}(\mathcal{U},M)$ and $\mu$ is $M$-\weak. Let $\Delta_{n+1} \coloneqq \{\varphi_1,\varphi_2,\ldots,\varphi_{s}\}$ be a finite collection of $\mathcal{L}(M)$-formulas with variables among $x_1,\ldots,x_n,y$. Let $p(x_1,\ldots,x_n,y) \coloneqq \bigwedge_{\varphi \in \Delta_{n+1}}\varphi$. Let $\Delta_{n}$ be the subcollection of formulas from $\Delta_{n+1}$ which contain variables only among $x_1,\ldots,x_n$. We let $r_{p}(x_1,\ldots,x_n) \coloneqq \bigwedge_{\varphi \in \Delta_n} \varphi$. 

Suppose that 
\begin{enumerate}
    \item $\mu_{y} \otimes \mu^{(n)}_{\bar{x}}(p(x_1,\ldots,x_n,y)) > 0$. 
    \item For $\mu^{(n)}$-almost all $q \in S_{\bar{x}}(M) \cap [r_{p}(\bar{x})]$, if $\bar{c} \models q$ then $\mu(p(\bar{c},y)) > 0$. In other words, 
    \begin{equation*}
        \mu^{(n)}|_{M} \left([r_p(\bar{x})] \cap \left( F_{\mu}^{\neg p} \right)^{-1}(\{1\}) \right) = 0.
    \end{equation*}
\end{enumerate}
 Then for any distinct indices $i_1 < \ldots < i_n$ and $t \geq 1$, we have that 
\begin{equation*}
    \mathbb{P}_{\mu} \left( \bigcup_{l > t} [r_{p}(x_{i_1},\ldots,x_{i_n}) \to p(x_{i_1},\ldots,x_{i_n},x_l)]\right) = 1.
\end{equation*} 
Additionally, if $\mu$ is $M$-\strong\ then for any pairwise distinct indices $j_1,\ldots,j_n$ and $t \geq 1$, we have that 
\begin{equation*}
    \mathbb{P}_{\mu} \left( \bigcup_{l  > t} [r_{p}(x_{j_1},\ldots,x_{j_n}) \to p(x_{j_1},\ldots,x_{j_n},x_l)]\right) = 1.
\end{equation*} 
\end{lemma}

\begin{proof} Fix variables $\bar{x}\coloneqq x_{i_1},\ldots,x_{i_n}$. Let $ m= \max\{i_1,\ldots,i_n,t\}$. For any $l \in \mathbb{N}$, we let
\begin{enumerate}
    \item $\theta_{l}(x_{l};x_1,\ldots,x_n) = r_{p}(x_1,\ldots,x_n) \wedge \neg p(x_1,\ldots,x_n,x_l)$. 
    \item $\psi_{l}(x_{l};x_1,\ldots,x_n) = \neg p(x_1,\ldots,x_n,x_{l})$. 
\end{enumerate}
Now we compute: 
\begin{align*}
&\mathbb{P}_{\mu}\left( \bigcup_{l  > t} \left[r_{p}(x_{i_1},\ldots,x_{i_n}) \to p(x_{i_1},\ldots,x_{i_n},x_l)\right]\right) \\
&\geq  \lim_{\substack{k \to \infty \\ k > m}} \mathbb{P}_{\mu} \left( \left[r_{p}(x_{i_1},\ldots,x_{i_n}) \to \bigvee_{l > m}^{k} p(x_{i_1},\ldots,x_{i_n},x_l)\right]\right)\\
&=1 -  \lim_{\substack{k \to \infty \\ k > m}} \mathbb{P}_{\mu} \left( \left[r_{p}(x_{i_1},\ldots,x_{i_n}) \wedge \bigwedge_{l > m}^{k} \neg p(x_{i_1},\ldots,x_{i_n},x_l)\right]\right) \\
&\overset{(a)}{=}1 -  \lim_{\substack{k \to \infty \\ k > m}} \mathbb{P}_{\mu} \left( \left[r_{p}(x_{1},\ldots,x_{n}) \wedge \bigwedge_{l > m}^{k} \neg p(x_{1},\ldots,x_{n},x_l)\right]\right) \\
&=1 -  \lim_{\substack{k \to \infty \\ k > m}} \mathbb{P}_{\mu} \left( \left[\bigwedge_{l > m}^{k} r_{p}(x_{1},\ldots,x_{n}) \wedge \neg p(x_{1},\ldots,x_{n},x_l)\right]\right) \\
&\overset{(b)}{=} 1 -  \lim_{\substack{k \to \infty \\ k > m}} \int_{ S_{\bar{x}}(M)}\prod_{l > m}^{k} F_{\mu_{l}}^{\theta_{l}} d\mu^{(n)}|_{M} \\
&\overset{(c)}{=} 1 -  \lim_{\substack{k \to \infty \\ k > m}} \int_{ S_{\bar{x}}(M)} \prod_{l > m}^{k} \left(  \mathbf{1}_{r_{p}(\bar{x})} \cdot F_{\mu_{l}}^{\psi_l} \right) d \mu^{(n)}|_{M} \\
 &\overset{(d)}{=}1 -  \lim_{\substack{k \to \infty \\ k > m}} \int_{ S_{\bar{x}}(M)} \left( \mathbf{1}_{r_{p}(\bar{x})}  \cdot F_{\mu}^{\psi} \right)^{k-m} d\mu^{(n)}|_{M}\\
 &\overset{(e)}{=}1 -  0 = 1
\end{align*}
We provide the following justifications: 
\begin{enumerate}[$(a)$]
    \item Renaming variables for clarity, see Fact \ref{fact:easy}. 
    \item Lemma \ref{label:conditional}.
    \item Direct computation. 
    \item Renaming variables, i.e.\ let $\psi(x;x_1,\ldots,x_n) = \neg p(x_1,\ldots,x_n,x)$. 
    \item Notice that 
        \begin{equation*}
            \mu^{(n)}|_{M} \left( \{ q \in S_{\bar{x}}(M) : \left(\mathbf{1}_{r_{p}(\bar{x})} \cdot F_{\mu}^{\psi}\right)(q) = 1 \} \right) =0
        \end{equation*}
        by condition (2). Hence we may apply Fact \ref{fact:converge}.
\end{enumerate}
We remark that the \emph{additionally} claim follows from the fact that if $\mu$ is $M$-\strong\ then the ordering of the variables does not matter. More explicitly, we have 
\begin{align*}
&\mathbb{P}_{\mu} \left( \bigcup_{l  > t} \left[r_{p}(x_{j_1},\ldots,x_{j_n}) \to p(x_{j_1},\ldots,x_{j_n},x_l) \right]\right) \\
&= \lim_{\substack{k \to \infty \\ k > t}} \mathbb{P}_{\mu} \left( \left[\bigvee_{l >t}^{k} r_{p}(x_{j_1},\ldots,x_{j_n}) \to p(x_{j_1},\ldots,x_{j_n},x_l) \right]  \right)\\
&\overset{(*)}{=} \lim_{\substack{k \to \infty \\ k > t}} \mathbb{P}_{\mu} \left( \left[\bigvee_{l >t}^{k} r_{p}(x_{1},\ldots,x_{n}) \to p(x_{1},\ldots,x_{n},x_l) \right]  \right)\\
&=\mathbb{P}_{\mu} \left( \bigcup_{l  > t} \left[r_{p}(x_{1},\ldots,x_{n}) \to p(x_{1},\ldots,x_{n},x_l)\right]\right) =1.
\end{align*}
Equation $(*)$ is an application of Fact \ref{fact:easy}. The last equality follows from the first portion of this lemma. 
\end{proof}

\begin{lemma}\label{lemma:ext2} Suppose $\mu \in \mathfrak{M}_{x}^{\inv}(\mathcal{U},M)$ and $\mu$ satisfies the \weak\ measure extension axiom. Then for any tuple $a_1,\ldots,a_{n+1} \in \mathcal{U}$ such that $\mu^{(n+1)}(\diam_{\bar{a},a_{n+1}}(\bar{x},x_{n+1})) >0$, $t \geq 1$, and increasing indices $i_1 < \ldots < i_n$, we have that 
\begin{equation*}
    \mathbb{P}_{\mu} \left(\bigcup_{l > t} [\diam_{\bar{a}}(x_{i_1},\ldots,x_{i_n}) \to \diam_{\bar{a},a_{n+1}}(x_{i_1},\ldots,x_{i_n},x_{l})]  \right) = 1.
\end{equation*}
If additionally $\mu$ satisfies the \strong\ measure extension axiom then for any sequence $j_1,\ldots,j_n$ of distinct indices (not necessarily increasing) we have that 
\begin{equation*}
        \mathbb{P}_{\mu} \left(\bigcup_{l > t} [\diam_{\bar{a}}(x_{j_1},\ldots,x_{j_n}) \to \diam_{\bar{a},a_{n+1}}(x_{j_1},\ldots,x_{j_n},x_{l})]  \right) = 1.
\end{equation*}
\end{lemma}

\begin{proof} Follows directly from Lemma \ref{lemma:ext}.
\end{proof}

\begin{theorem}\label{theorem:main} Suppose that $\mathcal{L}$ is a finite relational language. Let $\mu \in \mathfrak{M}_{x}^{\inv}(\mathcal{U},M)$ and suppose that $\mu$ satisfies the \strong\ measure extension axiom. Then there exists a countable $\mathcal{L}$-structure $N$ such that $\mathbb{P}_{\mu}(B_N) = 1$. 
\end{theorem}

\begin{proof} 
We describe a Borel subset $K$ of $S_{\mathbf{x}}(\mathcal{U})$ such that if $q, p \in K$, then $q_{\mathcal{L}} \cong p_{\mathcal{L}}$. After defining this Borel set, we construct an isomorphism via a back-and-forth argument. We let $S_{\bar{x}}^{\qf}(\emptyset)$ denote the collection of quantifier-free types over the empty set. For each $n$, we let $E_{n} \coloneqq \{q \in S_{x_1,\ldots,x_n}^{\qf}(\emptyset): \mu(\diam_{q}(x_1,\ldots,x_n))> 0\}$. Since $\mathcal{L}$ is a finite relational language, we have that $S_{x_1,\ldots,x_n}^{\qf}(\emptyset)$ is finite for each $n$ and as a consequence, $E_n$ is non-empty for each $n$. Consider the following Borel sets: 

\begin{enumerate}
    \item If $q \in E_1$ and $t \geq 1 $, we let 
\begin{equation*}
    C_{q,t} \coloneqq \bigcup_{l > t}  [\diam_{q}(x_l)]. 
\end{equation*}
By Proposition \ref{prop:inf-often}, $\mathbb{P}_{\mu}(C_{q,t}) = 1$.
    \item Suppose $q(x_1,\ldots,x_{n},x_{n+1}) \in E_{n+1}$, $\bar{j} = j_1,\ldots,j_n$ are distinct indices, and $t \geq 1$. Let $r_{q}$ be the restriction of $q$ to the variables $x_1,\ldots,x_n$. We define
    \begin{equation*}
        D_{q,\bar{j},t} = \bigcup_{l > t} [\diam_{r_{q}}(x_{j_{1}},\ldots,x_{j_{n}}) \to \diam_{q}(x_{j_{1}},\ldots,x_{j_{n}},x_{l})]
    \end{equation*}
By Lemma \ref{lemma:ext2}, $\mathbb{P}_{\mu}(D_{q,\bar{j},t}) = 1$. 
\end{enumerate}
Now define the set $K$ as follows: 
\begin{equation*}
    K \coloneqq \supp(\mathbb{P}_\mu) \cap \bigcap_{\substack{q \in E_1 \\ t \geq 1}} C_{q,t} \cap \bigcap_{n \geq 1} \left( \bigcap_{\substack{q \in E_{n+1} \\ j_1 \neq \ldots \neq j_n \\t \geq 1}} D_{q,\bar{j},t} \right).
\end{equation*}
Notice that $K$ is the countable intersection of subsets of full measures. Hence $\mathbb{P}_{\mu}(K) = 1$. We claim that the set $K$ witnesses the isomorphism property. We now prove this via a back-and-forth argument. Fix $p, q \in K$. Let $(a_i)_{i \geq 1} \models p$ and $(b_i)_{i \geq 1} \models q$. These are points in an elementary extension of our monster, namely $\mathcal{U}'$. It suffices to construct and isomorphism from the induced structure on $\{a_i\}_{i \geq 1}$ to the induced structure on $\{b_i\}_{i \geq 1}$. Consider the following
\begin{enumerate}
    \item (Forwards) Consider $a_1$. We have that $\mathbb{P}_{\mu}(\diam_{a_1}(x_1)) > 0$ because $\diam_{a_1}(x_1) \in p$ and $p\in \supp(\mathbb{P}_\mu)$. Hence $K \subseteq C_{\tp_{\qf}(a_1/\emptyset),1}$ and so $q \in C_{\tp_{\qf}(a_1/\emptyset),1}$. Hence $\diam_{a_1}(x_l) \in q$ for some $l > 1$. Let $l$ be the smallest such index and let $f(a_1) = b_{l}$. 
    \item (Backwards) Similar to above, find $a_{t}$ and send $f(a_{t}) = b_1$. 
    \item (Forwards) Suppose we are given $a_{i_1},\ldots,a_{i_{n}},b_{j_1},\ldots,b_{j_n}$ and a bijection $f:\{a_{i_1},\ldots,a_{i_{n}}\} \to \{b_{j_{1}},\ldots,b_{j_{n}}\}$ such that $f(a_{i_l}) = b_{j_{l}}$ and $\tp_{\qf}(a_{i_1},\ldots,a_{i_{n}}/\emptyset) = \tp_{\qf}(b_{j_1},\ldots,b_{j_{n}}/\emptyset)$. We let $\bar{a} = (a_{i_1},\ldots,a_{i_n})$. Let $l$ be the smallest index such that $a_{l}$ does not yet appear in $\{a_{i_1},\ldots,a_{i_{n}}\}$. Since $p \in \supp(\mathbb{P}_{\mu})$, it follows that 
    \begin{equation*}
        \mathbb{P}_{\mu}(\diam_{\bar{a},a_l}(x_{i_1},\ldots,x_{i_n},x_l)) > 0.
    \end{equation*}
    Since $\mu$ is $M$-\strong, by Fact \ref{fact:easy}, we have that 
    \begin{equation*}
        \mathbb{P}_{\mu}(\diam_{\bar{a},a_{l}}(x_{i_1},\ldots,x_{i_n},x_{l})) = \mathbb{P}_{\mu}(\diam_{\bar{a},a_{l}}(x_{1},\ldots,x_{n+1})),
    \end{equation*} 
and so $\tp_{\qf}(\bar{a},a_{l}/\emptyset) \in E_{n+1}$. Choose $t > \max\{i_1,\ldots,i_n,j_1,\ldots,j_n\}$. Then $K \subseteq D_{\tp_{\qf}(\bar{a},a_{l}/\emptyset),\bar{j},t}$ and so $q \in D_{\tp_{\qf}(\bar{a},a_{l}/\emptyset),\bar{j},t}$. By construction, this implies that there exists an index $l_* > t$ such that 
\begin{equation*}
    q \in [\diam_{\bar{a}}(x_{j_1},\ldots,x_{j_n}) \to \diam_{\bar{a},a_l}(x_{j_1},\ldots,x_{j_n},x_{l_*})].
\end{equation*}
Then 
\begin{equation*}
    \mathcal{U}' \models \diam_{\bar{a}}(b_{j_1},\ldots,b_{j_n}) \to \diam_{\bar{a},a_l}(b_{j_1},\ldots,b_{j_n},b_{l_*}).
\end{equation*}
Since $\tp_{\qf}(a_{i_1},\ldots,a_{i_n}/\emptyset) = \tp_{\qf}(b_{j_1},\ldots,b_{j_n}/\emptyset)$, we have that 
\begin{equation*}
    \mathcal{U}' \models \diam_{\bar{a}}(b_{j_1},\ldots,b_{j_n}).
\end{equation*}
By modus ponens, we conclude 
\begin{equation*}
    \mathcal{U}' \models \diam_{\bar{a},a_l}(b_{j_1},\ldots,b_{j_n},b_{l_*}).
\end{equation*}
We let $f(a_{l}) = b_{l_*}$. 
\item (Backwards) Similar. 
\end{enumerate}

We claim that $f$ is an isomorphism. By construction, $f$ is a bijection which preserves quantifier-free types over the empty set. 
\end{proof}

We remark that \cref{theorem:main} may fail if only the adequate extension axiom holds (e.g., see Example \ref{example:bad}).

\begin{theorem}\label{theorem:cat} Suppose that $\mathcal{L}$ is a finite relational language. Suppose that $\mu \in \mathfrak{M}_{x}^{\inv}(\mathcal{U},M)$, $\mu$ satisfies the \strong\ measure extension axiom, and $\mu$ has no realized part. Let $\mathbb{P}_{\mu}(B_{N})=1$. Then $\Th_{\mathcal{L}}(N)$ is countably categorical. 
\end{theorem}

\begin{proof} Since $\mu$ does not concentrate on finitely many realized types, we claim that $|N| = \aleph_0$. Recall the definition of $E_{n}$ from Theorem \ref{theorem:main}. Now consider the following set of axioms:
\begin{enumerate}
\item For each $q \in E_{n}$, we have $\exists x_1,\ldots,x_n q(x_1,\ldots,x_n)$.
    \item For each $q \in E_{n}$, $\forall x_1,\ldots,x_n ( \bigwedge_{1 \leq i < j \leq n}  x_i \neq x_j \to \bigvee_{q \in E_n} q(x_1,\ldots,x_n))$.
    \item For each $q \in E_{n+1}$, let $r_{q}$ be the restriction to the first $n$ variables. Then we have that 
    \begin{equation*}
        \forall x_1,\ldots,x_n \exists x_{n+1} \left( \bigwedge_{1 \leq i < j \leq n} x_i \neq x_j  \to ( r_{q}(x_1,\ldots,x_n) \to q(x_1,\ldots,x_n,x_{n+1})) \right).
    \end{equation*}
\end{enumerate}
By Theorem \ref{theorem:main} and Proposition \ref{prop:concentrate}, $N$ models the sentences above. Notice that given any two countable models, these axioms explicitly encode a back-and-forth argument. We conclude that our theory is $\aleph_0$-categorical. 
\end{proof}

Finally, our last observation stems from discussions with Ackerman, Freer, and Patel. It shows the first observed connection between labeled $\mathcal{L}$-structures and iterated Morley products of Keisler measures. We thank them for allowing us to add this theorem to our paper. 

\begin{theorem}\label{theorem:AFP}  Suppose that $\mathcal{L}$ is a finite relational language. Suppose moreover that $\mu \in \mathfrak{M}_{x}^{\inv}(\mathcal{U},M)$, $\mu$ has no realized part, and $\mu$ satisfies the excellent measure extension axiom. Let $N$ be the unique countable structure such that $\mathbb{P}_{\mu}(B_{N}) = 1$. Recall the map $g:S_{\mathbf{x}}^{+}(\mathcal{U}) \to \str_{\mathcal{L}}$ from Definition \ref{def:G_map}. Then the push-forward of $\mathbb{P}_{\mu}$ along $g$, denoted $g_{*}(\mathbb{P}_{\mu})$, is a measure on $\str_{\mathcal{L}}$ such that  
\begin{enumerate}
    \item $g_*(\mathbb{P}_{\mu})$ is invariant under the $\Sym(\mathbb{N})$-action on $\str_{L}$. 
    \item $g_*(\mathbb{P}_{\mu})$ concentrates on $N$, i.e. $g_*(\mathbb{P}_{\mu})(\{ M \in \str_{\mathcal{L}}: M \cong N\}) = 1$. 
\end{enumerate}
As a consequence of \cite[Theorem 1.1]{ackerman2016invariant}, the structure $N$ has trivial group-theoretic definable closure.
\end{theorem}

\begin{proof} We prove the statements. 
\begin{enumerate}
    \item By Fact \ref{fact:invariant-measures-agree}, it suffices to show that $g_*(\mathbb{P}_{\mu})(\llbracket \theta(i_1,\ldots,i_n) \rrbracket) = g_*(\mathbb{P}_{\mu})( \llbracket \theta(1,\ldots,n) \rrbracket)$ for any quantifier-free $\mathcal{L}$-formula $\theta(x_1,\ldots,x_n)$ and any sequence of distinct natural numbers $i_1,\ldots,i_n$. Since $\mu$ is excellent, we have that $\mu$ is self-associative and self-commutes. So, by Fact \ref{fact:easy} and Proposition \ref{prop:qf}, 
    \begin{align*}
    g_*(\mathbb{P}_{\mu})(\llbracket \theta(i_1,\ldots,i_n)\rrbracket) &= \mathbb{P}_{\mu}(\theta(x_{i_1},\ldots,x_{i_n})) \\ &= \mathbb{P}_{\mu}(\theta(x_{1},\ldots,x_{n}))  \\ &=   g_*(\mathbb{P}_{\mu})(\llbracket \theta(1,\ldots,n) \rrbracket). 
    \end{align*}
    \item We let $I(N) = \{M \in \str_{\mathcal{L}} : M \cong N\}$. Notice that $g^{-1}(I(N)) \supseteq B_{N} \cap S_{\mathbf{x}}^{+}(\mathcal{U})$. By Theorem \ref{theorem:main}, 
    \begin{equation*}
        g_*(\mathbb{P}_{\mu})(I_{N}) \geq  \mathbb{P}_{\mu}(B_{N}) = 1. \qedhere 
    \end{equation*}
\end{enumerate}
\end{proof}

\section{Examples}

In this section, we work with explicit examples. While the examples themselves are usually quite intuitive, the computations are sometimes a little complicated. We begin with two examples such that the induced structure on almost all random generic types is isomorphic to a fixed countable model. These are our friends from Example \ref{example:main}. In both cases, the measure of certain events requires integrating particular functions over type space. It seems that this step, the integration of these functions, is the critical step for working with specific examples. Different examples require different techniques.

\begin{example}\label{example:random} Let $\mathcal{U}$ be a monster model of the Rado graph and $M$ a small submodel. For $t \in (0,1)$, we let $\mu_t$ be the unique Keisler measure in $\mathfrak{M}_{x}(\mathcal{U})$ such that for any sequence of distinct tuples $a_1,\ldots,a_n,b_1,\ldots,b_m$, we have that
\begin{equation*}
    \mu_t \left(\bigwedge_{i=1}^{n} R(x,a_i) \wedge \bigwedge_{i = 1}^{m}  \neg R(x,b_i) \right) =  t^{n} \left( 1 - t\right)^{m}. 
\end{equation*} We note that for any $t \in (0,1)$, the measure $\mu_{t}$ is definable over the empty set and so $\mu_{t} \in \mathfrak{M}_{x}^{\inv}(\mathcal{U},M)$. We claim that the induced structure on almost all random generic types is isomorphic to the unique countable model of the Rado graph. In other words, if $N$ is the unique countable model of the Rado graph, then $\mathbb{P}_{\mu_{t}}(B_{N}) = 1$. 
\end{example}
\begin{proof}

First note that the induced substructure on any subset of a graph is also a graph. Hence for any type $p \in S_{\mathbf{x}}(\mathcal{U})$, by construction $p_{\mathcal{L}}$ is a graph. For each pair of natural numbers $n$ and subset $A$ of $[n] = \{1,\ldots,n\}$, we define the set 

\begin{equation*}
    K_{n,A} \coloneqq \bigcup_{l=1}^{\infty} \left[\bigwedge_{i \in A }R(x_i,x_l) \wedge \bigwedge_{j \in [n] \backslash A} \neg R(x_{j},x_l)\right]. 
\end{equation*}
For each pair $(n,A)$, we prove that $\mathbb{P}_{\mu}(K_{n,A}) = 1$. So fix a pair $(n,A)$. For each $l \in \mathbb{N}$, we let 
\begin{equation*}\psi_{l}(x_{l},\bar{x}) \coloneqq  \left( \bigvee_{i \in A} \neg R(x_i,x_l) \vee \bigvee_{j \in [n] \backslash A} R(x_{j},x_l) \right). 
\end{equation*} 
Now consider the following computation: 
\begin{align*}
    \mathbb{P}_{\mu_t}(K_{n,A})  &=\mathbb{P}_{\mu_t} \left( \bigcup_{l=1}^{\infty} \left[\bigwedge_{i \in A}^{n}R(x_i,x_l) \wedge  \bigwedge_{j \in [n] \backslash A}\neg R(x_{j},x_l)\right]  \right) \\
    &=\lim_{k \to \infty} \mathbb{P}_{\mu_t} \left( \left[\bigvee_{l=1}^{k} \left( \bigwedge_{i \in A}R(x_i,x_l) \wedge \bigwedge_{j \in [n] \backslash A} \neg R(x_{j},x_l) \right) \right]  \right)\\
    &=\lim_{k \to \infty} 1 - \mathbb{P}_{\mu_t} \left( \left[\bigwedge_{l=1}^{k} \left( \bigvee_{i \in A} \neg R(x_i,x_l) \vee \bigvee_{j \in [n] \backslash A} R(x_{j},x_l) \right) \right]  \right)\\
    &=1 - \lim_{k \to \infty} \mathbb{P}_{\mu_t} \left( \left[\bigwedge_{l=1}^{k} \left( \bigvee_{i \in A} \neg R(x_i,x_l) \vee \bigvee_{j \in [n] \backslash A} R(x_{j},x_l) \right) \right]  \right)\\
    &\geq 1 - \lim_{\substack{k > n \\ k \to \infty}} \mathbb{P}_{\mu_t} \left( \left[\bigwedge_{l >  n }^{k} \left( \bigvee_{i \in A} \neg R(x_i,x_l) \vee \bigvee_{j \in [n] \backslash A} R(x_{j},x_l) \right) \right]  \right)\\
    &\overset{(*)}{=}1 - \lim_{\substack{k > n \\ k \to \infty}} \int_{S_{\bar{x}}(M)} \prod_{l > n }^k F_{\mu_{l}}^{\psi_l} d\mu_t^{(n)} \\
    &\overset{(**)}{=}1 - \lim_{\substack{k > n \\ k \to \infty}} \int_{S_{\bar{x}}(M)} \prod_{l >  n}^{k}\left(1 - t^{|A|} \left( 1- t \right)^{n - |A|} \right) d\mu_t^{(n)} \\
    &=1 - \lim_{\substack{k > n \\ k \to \infty}} \int_{S_{\bar{x}}(M)} \left(1 - t^{|A|} \left( 1-t \right)^{n - |A|} \right)^{k - n} d\mu_t^{(n)} \\
    &=1 - \lim_{\substack{k > n \\ k \to \infty}} \left(1 - t^{|A|} \left( 1-t \right)^{n - |A|} \right)^{k - n} \\
    &= 1 - 0 = 1.
\end{align*}
Equation $(*)$ follows from Lemma \ref{label:conditional}. Equation $(**)$ follows from the fact that for each $l > n$, the map $F_{\mu}^{\psi_{l}}$ is constant. Indeed, if $q \in S_{\bar{x}}(M)$ and $(a_1,\ldots,a_n) \models q$ then 
\begin{align*}
    F_{\mu_{l}}^{\psi_{l}}(q) &= \mu_{l}\left(\bigvee_{i \in A} \neg R(a_i,x_l) \vee \bigvee_{j \in [n] \backslash A} R(a_{j},x_l) \right) \\
    &=1 - \mu_{l}\left(\bigwedge_{i \in A}  R(a_i,x_l) \wedge \bigwedge_{j \in [n] \backslash A} \neg R(a_{j},x_l)\right)\\
    &= \left(1 - t^{|A|} \left( 1-t  \right)^{n-|A|} \right).  
\end{align*}
We now consider the set $V = \bigcap_{(n,A) : n \in \mathbb{N}, A \subseteq [n]} K_{n,A}  $. $V$ is the intersection of countably many open sets of full measures and so $\mathbb{P}_{\mu_t}(V) = 1$. We claim that for any $p \in V$, the structure $p_{\mathcal{L}}$ is a countable model of the Rado graph. Since the Rado graph is countably categorical, we have that $p_{\mathcal{L}}$ is isomorphic to the unique countable model, $N$. Thus, $V \subseteq B_{N}$ and since $\mathbb{P}_{\mu_t}(V) = 1$, so $\mathbb{P}_{\mu_t}(B_{N}) = 1$. 
\end{proof}

We need a computation lemma before our next example. 

\begin{lemma}\label{lemma:computation} Recall Example \ref{example:main}. We have that $M = (\mathbb{R},<)$, $\mu_{L} \in \mathfrak{M}_{x}^{\inv}(\mathcal{U},M)$. Let $\st:S_{\bar{x}}(M) \to \mathbb{R}^{n} \cup \{*\}$ be the usual standard part map. 
Let $\tp:\mathbb{R}^{n} \to S_{\bar{x}}(M)$ be the map sending a tuple to its associated type.  
Suppose that $f: S_{x_1,\ldots,x_n}(M) \to [0,1]$ is a Borel function such that for any $p,q \in \supp(\mu_{L}^{(n)})$, if $\st(p) = \st(q)$ then $f(p) = f(q)$. Then 
\begin{equation*}
    \int_{S_{\bar{x}}(M)} f d \mu_{L}^{(n)} = \int_{[0,1]^{n}} (f \circ \tp)  dL^{n}.
\end{equation*}
\end{lemma} 

\begin{proof} Standard push-forward/quotient argument and the fact that every definable subset of $M^{n}$ is Borel. 
\end{proof}

\begin{example}\label{example:DLO} Let $\mathcal{U}$ be a monster model of DLO and $\mathbb{R} = M \prec \mathcal{U}$. Let $L$ be the Lebesgue measure restricted to the interval $[0,1]$. We recall from Example \ref{example:main} the measure $\mu_{L}$ in $\mathfrak{M}_{x}^{\inv}(\mathcal{U},M)$. For any formula $\varphi(x) \in \mathcal{L}_{x}(\mathcal{U})$, we have that
\begin{equation*}
    \mu_{L}(\varphi(x)) = L(\{ r \in \mathbb{R} : \mathcal{U} \models \varphi(r)\}). 
\end{equation*}
Let $N$ be the unique countable model of DLO, i.e.\ $N \cong (\mathbb{Q},<)$. Then $\mathbb{P}_{\mu_{L}}(B_{N}) = 1$. In other words, the induced structure on almost all random generic types is isomorphic to $(\mathbb{Q},<)$. 
\end{example}

\begin{proof} For notation simplicity, in this proof we set $\mu = \mu_L$. We will also write $\mu_{x_{t}}$ simply as $\mu_{t}$ for $t \geq 1$.  The induced substructure on any subset of a total ordering is a total order. Hence it suffices to show that the induced structure on almost all random generic types is a dense ordering and is without endpoints. We first show the ``without endpoints'' condition: For any $n \geq 1$, we consider the sets
\begin{equation*}
    A_n \coloneqq \left( \bigcup_{l=1}^{\infty} [x_n < x_l] \right) \text{ and } F_n \coloneqq\left( \bigcup_{l=1}^{\infty} [x_n > x_l] \right)
\end{equation*}
We show that $\mathbb{P}_{\mu}(A_n) = 1$. A similar computation shows that $\mathbb{P}_{\mu}(F_n) =1$. Indeed, let $\psi_l(x_{n},x_l) \coloneqq x_n \geq x_l$. Now consider the following computation:
\begin{align*}
\mathbb{P}_{\mu}(A_n) &= 
    \mathbb{P}_{\mu} \left( \bigcup_{l=1}^{\infty} [x_n < x_l] \right) \\&= \lim_{k \to \infty} \mathbb{P}_{\mu} \left( \left[\bigvee_{l = 1}^{k} x_n < x_l \right] \right)\\
    &\geq  \lim_{\substack{k > n \\ k \to \infty}}\mathbb{P}_{\mu} \left( \left[\bigvee_{l > n}^{k} x_n < x_l \right] \right)\\
    &= 1 - \lim_{\substack{k > n \\ k \to \infty}}\mathbb{P}_{\mu} \left( \left[\bigwedge_{l > n}^{k} x_n \geq  x_l \right] \right)\\
     &\overset{(a)}{=} 1 - \lim_{\substack{k > n \\ k \to \infty}} \int_{S_{x_{n}}(M)} \prod_{l > n}^{k} F_{\mu_l}^{\psi_l} d\mu_n\\
     &\overset{(b)}{=}1 - \lim_{\substack{k > n \\ k \to \infty}} \int_{p \in S_{x_{n}}(M)}\st(p)^{k-n} d\mu_n\\
     &\overset{(c)}{=}1 - \lim_{\substack{k > n \\ k \to \infty}} \int_{[0,1]} x^{k-n} dL \\
     &=1 - \lim_{\substack{k > n \\ k \to \infty}} \frac{1}{k-n + 1}\\
     &= 1 - 0 = 1. 
\end{align*}
We provide the following justifications: 
\begin{enumerate}[($a$)]
    \item Lemma \ref{label:conditional}.
    \item Follows from a direct computation. Notice that if $q \in \supp(\mu_{n})$ and $a \models q$, then 
\begin{equation*}
    \prod_{l > n}^{k} F_{\mu_l}^{\psi_{l}}(q) = \prod_{l > n}^{k}\mu_{l}(a \geq x_{l}) = \st(a)^{k-n} = \st(q)^{k-n}.   
\end{equation*}
    \item Lemma \ref{lemma:computation}. 
\end{enumerate}
We now prove the density claim. For any $n,m \geq 1$, we consider the sets 
\begin{equation*}
    C_{n,m} \coloneqq \bigcup_{l =1}^{\infty} \left[ ( x_n < x_m) \to (x_n < x_l < x_m)   \right],
\end{equation*}
and 
\begin{equation*}
    D_{n,m} \coloneqq \bigcup_{l =1}^{\infty} \left[ ( x_n > x_m) \to (x_m < x_l < x_n)   \right],
\end{equation*}
Fix $n,m$ and let $h = \max\{m,n\}$. Here we let $\psi_{l}(x_{l};x_n,x_m) \coloneqq x_n < x_m \wedge \neg (x_n < x_l < x_m)$. Now Behold! 

\begin{align*} 
\mathbb{P}_{\mu}(C_{n,m}) &=\mathbb{P}_{\mu}  \left( \bigcup_{l =1}^{\infty} \left[ ( x_n < x_m) \to (x_n < x_l < x_m)   \right] \right) \\
& = \lim_{k \to \infty} \mathbb{P}_{\mu}\left( \left[x_n < x_m \to   \bigvee_{l=1}^{k} x_n < x_l < x_m  \right] \right) \\
& \geq \lim_{\substack{k \to \infty \\ k > \max\{m,n\}}} \mathbb{P}_{\mu}\left( \left[x_n < x_m \to   \bigvee_{l > \max\{m,n\}}^{k} x_n < x_l < x_m  \right] \right) \\
& \geq \lim_{\substack{k \to \infty \\ k > h}} 1- \mathbb{P}_{\mu}\left( \left[x_n < x_m \wedge \bigwedge_{l > h}^{k} \neg (x_n < x_l < x_m)  \right] \right) \\
& \overset{(a)}{=} 1 - \lim_{\substack{k \to \infty \\ k >h}}\int_{S_{x_nx_m}(M)} \prod_{l > h}^{k} F_{\mu_l}^{\psi_l} d(\mu_{n} \otimes \mu_{m}) \\
&\overset{(b)}{=}  1 - \lim_{\substack{k \to \infty \\ k >h}}\int_{p \in S_{x_nx_m}(M)} \mathbf{1}_{x_n < x_m}(p)   (1 - \st_{m}(p) + \st_{n}(p))^{k - h}  d(\mu_{n} \otimes \mu_{m}) \\
&\overset{(c)}{=}  1 - \lim_{\substack{k \to \infty \\ k >h}}\int_{(x,y) \in [0,1]^2 : x < y}  (1 - y + x)^{k - h}  dL^{2}\\
& \overset{(d)}{=} 1 - 0 = 1. 
\end{align*}

We provide the following justifications: 
\begin{enumerate}[($a$)]
    \item Lemma \ref{label:conditional}.
    \item Lemma \ref{lemma:computation}. 
    \item Fix $s \in S_{x_n,x_m}(M)$ and let $(a_n,a_m) \models s$. 
\begin{align*}
    \prod_{l > h}^{k} F_{\mu_l}^{\psi_{l}}(s) &= \prod_{l > h}^{k} \mu_{l}( a_n < a_m \wedge \neg (a_n < x_l < a_m))  \\ 
    &= \prod_{l > h}^{k} \left(  \mathbf{1}_{x_n < x_m} (a_n,a_m)  \cdot (1 - \mu_{l}(a_n < x_l < a_m) \right) \\
    &= \prod_{l > h}^{k} \left(  \mathbf{1}_{x_n < x_m} (a_n,a_m)  \cdot (1 - (\st(a_m) - \st(a_n)))  \right) \\
    &= \mathbf{1}_{x_n < x_m}(a_n,a_m) \prod_{l > h}^{k}  (1 - \st(a_m) + \st(a_n))   \\
    &= \mathbf{1}_{x_n < x_m}(a_n,a_m)   (1 - \st(a_m) + \st(a_n))^{k - h}   \\
    &= \mathbf{1}_{x_n < x_m}(s)   (1 - \st_{m}(s) + \st_{n}(s))^{k - h}. 
\end{align*}
\item Notice that if $g : \{(x,y) \in [0,1]: x < y\} \to [0,1]$ via $g(x,y) = 1 - y + x$, then the image of $g$ is a subset of $[0,1)$. Apply Fact \ref{fact:converge}. 
\end{enumerate}

A similar computation shows that $\mathbb{P}_{\mu}(D_{n,m}) = 1$. Consider the Borel set 
\begin{equation*}
        K = \bigcap_{n \geq 1} A_n \cap \bigcap_{n \geq 1} F_n \cap \bigcap_{n,m \geq 1} C_{n,m} \cap \bigcap_{n,m \geq 1} D_{n,m}. 
\end{equation*}
$K$ is the countable intersection of sets of full measures and therefore $\mathbb{P}_{\mu}(K) = 1$. We claim that if $p \in K$ then $p_{\mathcal{L}} \cong (\mathbb{Q},<)$. 
\end{proof}

The next example is easier to deal with. We provide it as a contrast to the examples in the next subsection. 

\begin{example} Consider $M = (2^{< \omega}, \leq)$ where $M \models \tau \leq \sigma$ if and only if $\tau$ is an initial segment of $\sigma$. Let $\mathcal{U}$ be a monster model of $T$. There is a unique measure $\mu \in \mathfrak{M}_{x}(\mathcal{U})$ such that for any $\tau \in M$, $\mu( \tau \leq x) = \frac{1}{2^{|\tau|}}$ and $\mu(x = \tau) = 0$. We claim that $\mu$ is $M$-\strong\ (even more, it is smooth). Now let $N$ be the unique countable model of an infinite anti-chain in this language. We claim that $\mathbb{P}_{\mu}(B_{N}) = 1$. 
\end{example}

\begin{proof}
Choose distinct $n,m \geq 1$. Since the measure $\mu$ is $M$-smooth, it self-commutes and so  $\mu_{x_n} \otimes \mu_{x_m} = \mu_{x_m} \otimes \mu_{x_n}$. Let $\varphi(x_{n},x_{m}) = x_{n} < x_{m}$ and notice 
\begin{align*}
    \mathbb{P}_{\mu}([x_n < x_m]) &= (\mu_{x_n} \otimes \mu_{x_m})(x_n < x_m)\\
    &= \int_{S_{x_m}(M)} F_{\mu_{x_n}}^{\varphi} d\mu_{x_n} \\
    &\overset{(*)}{=} \int_{S_{x_m}(M)}0 d\mu_{x_m} \\
    &= 0.\\
\end{align*}
We justify equation $(*)$. Let $p \in \supp(\mu_{x_m})$. Then there exists a unique path $\gamma_{p} \in 2^{\omega}$ such that for each initial segment $\tau$ of $\gamma_{p}$, $\tau \leq x \in p$. We let $\gamma_{p}(i)$ be the truncated path at height $i$. Let $a \models p$ and notice that

\begin{align*}
    F_{\mu_{x_n}}^{\varphi}(p) &= \mu(x_n < a)  \\
    & \overset{(*)}{\leq} \lim_{\ell \to \infty}\mu_{x_n} \left( \bigvee_{i=0}^{\ell} x_n = \gamma_{p}(i) \vee \gamma_{p}(\ell) \leq x_n \right)\\
    &\overset{(**)}{\leq} \lim_{\ell \to \infty} \mu_{x_m}(\gamma_{p}(\ell) \leq x_m)  \\
    &= \lim_{\ell \to \infty} \frac{1}{2^{\ell}} = 0. 
\end{align*}
Inequality $(*)$ follows from set containment while inequality $(**)$ follows from the fact that finite sets have measure $0$ with respect to $\mu$. 
We claim that 
\begin{equation*}
    K = \bigcap_{n,m \geq 1} [\neg (x_n < x_m)]
\end{equation*}
has measure one and if $p \in K$, $p_{\mathcal{L}} \cong N$. 
\end{proof}

\subsection{The bad} We now consider examples which answer a resounding no to Question \ref{question:main}. Our first example is in a finite relational language. We remark that the measure satisfies the \weak\ measure extension axiom. The measure is also quite simple: a measure which concentrates on two types. However, the measure fails to self commute and therefore it does not satisfy the \strong\ measure extension axiom. This obstruction arises from the asymmetry of the Morley product.

\begin{example}\label{example:bad} Consider the structure $M = (\mathbb{R} \times \{0,1\}; P,<)$.  We have that $M \models P((a,i))$ if and only if $i = 0$, and we have that $M \models (a,i) < (b,j)$ if and only if $a < b$. Let $\mathcal{U}$ be a monster model such that $M \prec \mathcal{U}$. Consider the two definable types $q_1,q_2$ where
\begin{enumerate}
    \item $P(x) \in q_1$ and $\{k < x: k \in \mathcal{U} \} \subseteq q_1$. 
    \item $\neg P(x) \in q_2$ and $\{k < x: k \in \mathcal{U} \} \subseteq q_2$.
\end{enumerate}
Consider the measure $\mu = \frac{1}{2}\delta_{q_1} + \frac{1}{2}\delta_{q_2}$. Then for any countable $\mathcal{L}$-structure $N$, we have that $\mathbb{P}_{\mu}(B_{N}) = 0$.  
\end{example}

\begin{proof} We need only consider $N$ for which there exists $p \in \supp(\mathbb{P}_{\mu})$ with $p_{\mathcal{L}} \cong N$. Fix $p \in \supp(\mathbb{P}_{\mu})$. We claim that $p|_{<}$ is isomorphic to $(\mathbb{N},<)$ via $x_i \to i$. Indeed, by $(3)$ of Proposition \ref{prop:bigsupport}, we have that $p = \bigotimes_{\ell < \omega} r_\ell(x_\ell)$ where $r_\ell \in \{q_1,q_2\}$. Moreover, by Proposition \ref{prop:concentrate}, we have that $p \vdash x_i \neq x_j$ whenever $i \neq j$. We only need to argue that $x_i < x_j \in p$ if and only if $i < j$. If $i < j$, notice that 
\begin{align*}
    (x_i < x_j) \in p &\Longleftrightarrow (x_i < x_j) \in \bigotimes_{\ell < \omega} r_{\ell} (x_\ell) \\ 
    &\Longleftrightarrow (x_i < x_j) \in r_j(x_j) \otimes r_i(x_i) \\ 
    &\Longleftrightarrow (a < x_j) \in r_j(x_j) \text{ where } a \models r_i|_{M}, 
\end{align*}
and the last line is true since $r_i \in \{q_1,q_2\}$ and the formula $(a< x_j)$ is in both of these types. 

Now suppose that $i$ is not greater than $j$. First, notice that if $i = j$, then $x_i < x_j \not \in p$ since $\mathcal{U} \models \forall x \neg (x < x)$. Now, if $i > j$, we have that,
\begin{align*}
    (x_i < x_j) \in p &\Longleftrightarrow (x_i < x_j) \in r_i(x_i) \otimes r_j(x) \\ 
    &\Longleftrightarrow (x_i < b) \in r_i(x_i) \text{ where } b \models r_j|_M, 
\end{align*}
and the last line is false since $r_i \in \{q_1,q_2\}$ and $(x_i < b)$ is in neither of the types.

For every $i \geq 1$, either $P(x_i) \in p$ or $\neg P(x_i) \in p$. For each $\tau \in 2^{<\omega}$, consider the formula given by
\begin{align*}
    A_{\tau}(x_1,\ldots,x_{|\tau|}) &\coloneqq \left( \bigwedge_{\tau(i) = 1}P(x_i) \wedge \bigwedge_{\tau(i) = 0} \neg P(x_i)\right). 
\end{align*}
A standard computation shows that 
\begin{equation*}
    \mathbb{P}_{\mu}( A_{\tau}(x_1,\ldots,x_{|\tau|})) = \frac{1}{2^{|\tau|}}.
\end{equation*}
 Notice that if $\tau$ and $\sigma$ are not initial segments of one another, then $[A_{\tau}(x_1,\ldots,x_{|\tau|})] \cap [A_{\sigma}(x_1,\ldots,x_{|\sigma|})] = \emptyset$. Furthermore, there exists some path $\gamma_{p}$ in $2^{\omega}$ such that
\begin{align*}
    p_{\mathcal{L}} \models \exists x_1,\ldots,x_{|\tau|}\left( A_{\tau}(x_1,\ldots,x_{|\tau|}) 
    \wedge \forall z (x_1 \leq z) \wedge \left(  \bigwedge_{i=1}^{|\tau|-1} \forall z (x_i < z \to x_{i+1}\leq z) \right) \right),
\end{align*}
if and only if $\tau$ is an initial segment of $\gamma_{p}$. So for any $p \in \supp(\mathbb{P}_{\mu})$, we have that 
\begin{equation*}
    \mathbb{P}_{\mu}(B_{p_{\mathcal{L}}}) \leq \lim_{i \to \infty} \mathbb{P}_{\mu}([A_{\gamma_p(i)}(x_1,\ldots,x_{i})]) = \lim_{i \to \infty } \frac{1}{2^{i}} = 0,
\end{equation*}
where $\gamma_{p}(i)$ is the path $\gamma_{p}$ truncated at height $i$. 
\end{proof}

The next example gives a measure which satisfies the \strong\ measure extension axiom but over an infinite language. We remark that the underlying theory is stable. This obstruction seems to arise because the type space is \emph{too large}. There are continuum many 1-types over the empty set which the measure equidistributes over, and so choosing a \emph{countable} subcollection of types cannot yield a unique model up to isomorphism.

\begin{example}\label{example:bad2} Consider the structure $M = ([0,1],(B_{p,q}(x))_{q,p \in \mathbb{Q} \cap [0,1]})$ where 
\begin{equation*}
    M \models B_{p,q}(a) \Longleftrightarrow |a- p| < q.  
\end{equation*}
There exists a unique measure $\mu \in \mathfrak{M}_{x}^{\inv}(\mathcal{U},M)$ determined by 
\begin{equation*}
    \mu(B_{p,q}(x)) = L(\{r \in [0,1]: M \models B_{p,q}(r) \})
\end{equation*}
where $L$ is the Lebesgue measure. Notice that for every $p \in \supp(\mathbb{P}_{\mu})$, we have that $p_{\mathcal{L}}$ realizes a countable collection of cuts. Notice moreover that if $p_{\mathcal{L}}$ and $q_{\mathcal{L}}$ realize different cuts, then they are not isomorphic. Realizing a specific cut $c$ has probability $0$ and so the collection of random generic types which do not realize $c$ has probability 1. So for any $p \in \supp(\mathbb{P}_{\mu})$ and any cut $c$ realized by $p$, we have that $\mathbb{P}_{\mu}(B_{p_{\mathcal{L}}}) \leq \mathbb{P}_{\mu}( \{q \in S_{\mathbf{x}}(\mathcal{U}): q_{\mathcal{L}}$ realizes $c\} )= 0$. This gives the result. 
\end{example}

\section{Witnessing dividing lines}

Throughout this section we assume that $\mathcal{L}$ is countable, but not necessarily relational.

The second kind of events we are interested in are those related to \emph{witnessing dividing lines}. There are many dividing lines, but in this section we focus on three: instability, the independence property, and the strict order property. We are motivated by the following \emph{soft question}:  

\begin{question} Given a random generic type, does it witness a dividing line? 
\end{question}

There are several ways to interpret the question. Under our interpretation, we show that these events satisfy a $0$-$1$ law. As an application, we prove that if $\mu$ is \emph{fim}, then $\mathbb{P}_{\mu}$-almost no random generic types witness instability, and hence neither witness the independence property nor the strict order property. Since \emph{fim} measures in NIP theories are called generically stable, these results reinforce the view that \emph{fim} measures outside NIP should be regarded as generically stable as well. In the next subsection, we turn to the local NIP setting and study certain permutation averages.

The next fact follows directly from the observation that the product of measures which are both definable and finitely satisfiable is again definable and finitely satisfiable (see \cite{NIP3}). 

\begin{fact}\label{fact:overkill} Suppose that $\mu \in \mathfrak{M}_{x}^{\inv}(\mathcal{U},M)$. If $\mu$ is \emph{fim}, then $\mu^{(m)}$ is definable and finitely satisfiable in $M$ for every $m \geq 1$. Recall that a measure $\lambda \in \mathfrak{M}_{x_1,\ldots,x_m}(\mathcal{U})$ is finitely satisfiable in $M$ if for every $\mathcal{L}(\mathcal{U})$-formula, if $\lambda(\varphi(x_1,\ldots,x_m)) > 0$ then there exists $a_1,\ldots,a_m$ in $M$ such that $\mathcal{U} \models \varphi(a_1,\ldots,a_m)$. 
\end{fact}

We now define the events which pertain to instability, the independence property, and the strict order property. 

\begin{definition} Fix a formula $\varphi(x,y) \in \mathcal{L}_{xy}(\mathcal{U})$. The following formulas correspond to \emph{witnessing $n$-instability of $\varphi$}, \emph{shattering a set of size $n$ via $\varphi$}, and \emph{witnessing a strict order of length $n$ for $\varphi$} respectively.
\begin{enumerate}
    \item $O^{\varphi}_n(x_1,\ldots,x_n) \coloneqq \exists y_1,\ldots,y_n \left( \bigwedge_{1 \leq i \leq j \leq n} \varphi(x_i,y_j) \wedge \bigwedge_{1 \leq j < i \leq n} \neg \varphi(x_i,y_j) \right)$,
    \item $I_{n}^{\varphi}(x_1,\ldots,x_n) \coloneqq \exists y_{ \emptyset},\ldots,y_{\mathcal{P}(n)}  \bigwedge_{K \in \mathcal{P}(n)}\left( \bigwedge_{i \in K} \varphi(x_i,y_K) \wedge \bigwedge_{i \not \in K} \neg \varphi(x_i,y_K) \right)$,
    \item $L_{n}^{\varphi}(x_1,\ldots,x_n) \coloneqq \forall y (\varphi(x_1,y) \subsetneq \ldots \subsetneq \varphi(x_n,y))$. 
\end{enumerate}
The following events correspond to \emph{witnessing the instability of $\varphi$}, \emph{witnessing the independence property for $\varphi$} and \emph{witnessing the strict order property for $\varphi$}. 
\begin{enumerate}
    \item $\mathbf{O}^{\varphi} = \bigcap_{n=1}^{\infty} [O_{n}^{\varphi}(x_{1},\ldots,x_{n})]$,
    \item $\mathbf{I}^{\varphi} = \bigcap_{n=1}^{\infty} [I_{n}^{\varphi}(x_{1},\ldots,x_{n})]$, 
    \item $\mathbf{L}^{\varphi} = \bigcap_{n = 1}^{\infty} [L_{n}^{\varphi}(x_1,\ldots,x_n)]$. 
\end{enumerate}
The following events correspond to \emph{witnessing instability}, \emph{witnessing IP}, and \emph{witnessing the strict order property}. 
\begin{enumerate}
    \item $\mathbf{O} = \bigcup_{\varphi \in \mathcal{L}} \mathbf{O}^{\varphi}$
    \item $\mathbf{I} = \bigcup_{\varphi \in \mathcal{L}} \mathbf{I}^{\varphi}$
    \item $\mathbf{L} = \bigcup_{\varphi \in \mathcal{L}} \mathbf{L}^{\varphi}$
\end{enumerate}
\end{definition}

The following facts are either immediate or left to the reader as an exercise. We will sometimes use these results without reference.

\begin{fact}\label{fact:basic}  Let $\mu \in \mathfrak{M}_{x}^{\inv}(\mathcal{U},M)$ and suppose that $\mu$ is $M$-\weak. Fix any formula $\varphi(x,y) \in \mathcal{L}_{xy}(\mathcal{U})$. If $R \in \{O,I,L\}$,
\begin{enumerate}
    \item For any $n < m$, $[R_{m}^{\varphi}(x_1,\ldots,x_m)] \subseteq [R_{n}^{\varphi}(x_1,\ldots,x_n)]$
    \item $\mathbf{L}^{\varphi} \subseteq \mathbf{O}^{\varphi}$ and so $\mathbf{L} \subseteq \mathbf{O}$.
    \item $\mathbf{I}^{\varphi} \subseteq \mathbf{O}^{\varphi}$ and so $\mathbf{I} \subseteq \mathbf{O}$.
    \item $\mathbb{P}_{\mu}(\mathbf{R}^{\varphi}) = \lim_{n \to \infty} \mu^{(n)}(R_{n}^{\varphi}(x_1,\ldots,x_n))$. 
    \item $\mathbb{P}_{\mu}(\mathbf{R}) \leq \sum_{\varphi \in \mathcal{L}} \mathbb{P}_{\mu}(\mathbf{R}^{\varphi})$.  
\end{enumerate}
\end{fact}

\begin{lemma}\label{lem:01} Let $\mu \in \mathfrak{M}_{x}^{\inv}(\mathcal{U},M)$ and suppose that $\mu$ is $M$-\weak. If $\varphi(x,y) \in \mathcal{L}_{xy}(\mathcal{U})$, then $\mathbb{P}_{\mu}(\mathbf{O}^{\varphi}), \mathbb{P}_{\mu}(\mathbf{I}^{\varphi}), \mathbb{P}_{\mu}(\mathbf{L}^{\varphi}) \in \{0,1\}$. 
\end{lemma}
\begin{proof}
    We prove the statement for $\mathbf{O}^{\varphi}$. A similar proof works for the other dividing lines. We have two cases:
    \begin{enumerate}
        \item $\mu^{(n)}(O_{n}^{\varphi}(x_1,\ldots,x_n)) = 1$ for every $n \geq 1$. Then clearly $\mathbb{P}_{\mu}(\mathbf{O}^{\varphi}) = 1$. 
        \item Otherwise, let $k$ be the smallest integer such that $\mu^{(k)}(O_{k}^{\varphi}(x_1,\ldots,x_k)) < 1$. Notice that the sequence $(\mu^{(n)}(O^{\varphi}_n(x_1,\ldots,x_k)))_{n \geq 1}$ is a decreasing sequence bounded below by $0$. Hence it converges. Thus it suffices to find a subsequence which converges to $0$. Notice that
    \begin{align*}
    \mu^{(mk)}(O^{\varphi}_{mk}(x_1,\ldots,x_{mk})) &\leq \mu^{(mk)}\left( \bigwedge_{i=1}^{m} O^{\varphi}_k(x_{(i-1)\cdot k + 1},\ldots,x_{i \cdot k}) \right)\\
    &\overset{(*)}{=}  \left( \mu^{(k)}(O_{k}(x_1,\ldots,x_k)) \right)^{m} \to 0,
    \end{align*}
    as $m$ goes to infinity. Equation $(*)$ follows by  \cref{lemma:product}. \qedhere
\end{enumerate}
\end{proof}

\begin{proposition}\label{prop:main} Let $\mu \in \mathfrak{M}_{x}^{\inv}(\mathcal{U},M)$ and $\varphi(x,y) \in \mathcal{L}_{xy}(\mathcal{U})$. If $\mu$ is \emph{fim} then $\mathbb{P}_{\mu}(\mathbf{O}^{\varphi}) = \mathbb{P}_{\mu}(\mathbf{I}^{\varphi})  = \mathbb{P}_{\mu}(\mathbf{L}^{\varphi}) = 0$. 
\end{proposition}
\begin{proof} By $(2)$ and $(3)$ of Fact \ref{fact:basic}, it suffices to show that $\mathbb{P}_{\mu}(\mathbf{O}^{\varphi}) = 0$. By Lemma \ref{lem:01}, it suffices to show that $\mathbb{P}_{\mu}(\mathbf{O}^{\varphi}) \neq 1$. Towards a contradiction, assume $\mathbb{P}_{\mu}(\mathbf{O}^{\varphi}) = 1$. This implies that $\mu^{(n)}(O_{n}^{\varphi}(x_1,\ldots,x_n)) =1$ for each $n \geq 1$. Since $\mu$ is \emph{fim}, there exists a formula $\theta_{n}(x_1,\ldots,x_n)$ such that 
\begin{enumerate}
    \item $\mu^{(n)}(\theta_{n}(x_1,\ldots,x_n)) \approx_{\frac{1}{100}}  1$
    \item If $\models \theta_{n}(\bar{a})$, then 
    \begin{equation*}
        \mu(\varphi(x,c)) \approx_{\frac{1}{100}} \Av(\bar{a})(\varphi(x,c)). 
    \end{equation*}
\end{enumerate}
Now notice that since $\mathbb{P}_{\mu}([O_{2n}^{\varphi}(x_1,\ldots,x_{2n})]) = 1$, 
\begin{align*}
    \mathbb{P}_{\mu}\Big( O_{2n}^{\varphi}(x_{1},\ldots,x_{2n}) \wedge \theta_{n}(x_1,\ldots,x_n) \wedge &\theta_n(x_{n+1},\ldots,x_{2n})\Big) \\
    &= \mathbb{P}_{\mu}\left(  \theta_{n}(x_1,\ldots,x_n) \wedge \theta_n(x_{n+1},\ldots,x_{2n})\right)  \\ 
    &= \mathbb{P}_{\mu}\left(  \theta_{n}(x_1,\ldots,x_n) \right) \cdot \mathbb{P}_{\mu} \left(\theta_n(x_{n+1},\ldots,x_{2n})\right) \\ 
    &= \left(1- \frac{1}{100}\right)^{2} > 0. 
\end{align*}
By Fact \ref{fact:overkill} the measure $\mu^{(2n)}$ is finite satisfiable in $M$. Therefore there exists some tuple $(\bar{a},\bar{b})$ in $M$ such that $\mathcal{U} \models O_{2n}^{\varphi}(\bar{a},\bar{b}) \wedge \theta_n(\bar{a}) \wedge \theta_{n}(\bar{b})$. Hence there exists some $c \in M$ such that $\mathcal{U} \models \bigwedge_{i \leq n} \varphi(a_i,c) \wedge \bigwedge_{i \leq n} \neg \varphi(b_i,c)$. However, this implies that $\Av(\bar{a})(\varphi(x,c)) = 1$ while $\Av(\bar{b})\varphi(x,c)) = 0$. In particular,  
\begin{align*}
    1 &= \Av(\bar{a})(\varphi(x,c)) \\
    &\approx_{\frac{1}{100}} \mu(\varphi(x,c)) \\ 
    &\approx_{\frac{1}{100}} \Av(\bar{b})(\varphi(x,c)) \\ 
    &= 0 \\ 
    &\Longrightarrow 1 \approx_{\frac{2}{100}} 0. 
\end{align*}
This is a contradiction.  
\end{proof}

\begin{theorem}\label{theorem:witness} Let $\mu \in \mathfrak{M}_{x}^{\inv}(\mathcal{U},M)$. If $\mu$ is \emph{fim}, then $\mathbb{P}_{\mu}(\mathbf{O})= \mathbb{P}_{\mu}(\mathbf{I})= \mathbb{P}_{\mu}(\mathbf{L})= 0 $.
\end{theorem}

\begin{proof} By Fact \ref{fact:basic} and Proposition \ref{prop:main}, 
\begin{equation*}
    \mathbb{P}_{\mu}(\mathbf{I}) \leq  \mathbb{P}_{\mu}(\mathbf{O}) \leq \sum_{\varphi \in \mathcal{L}} \mathbb{P}_{\mu}(\mathbf{O}^{\varphi}) = \sum_{\varphi \in \mathcal{L}} 0 = 0. 
\end{equation*}
Replacing $\mathbf{I}$ with $\mathbf{L}$ gives the other computation. 
\end{proof}

\subsection{Extension and alternate interpretations}

As mentioned earlier, the definitions of $\mathbf{O}^{\varphi},\mathbf{I}^{\varphi}$ and $\mathbf{L}^{\varphi}$ are \emph{on-the-nose}. Theorem \ref{theorem:witness} holds in a slightly more general setting. We may replace \emph{witnessing} with \emph{eventually witnessing} and derive a similar result. We remark that the proof is quite similar to the proof of Theorem \ref{theorem:witness} and so we provide the definitions and give a sketch of the argument. 

\begin{definition} Fix a formula $\varphi(x,y) \in \mathcal{L}_{xy}(\mathcal{U})$. For any $t \geq 1$ and $R \in \{O,I,L\}$ we let 
\begin{enumerate}
    \item $\mathbf{R}_{t}^{\varphi} \coloneqq \bigcap_{k \geq 0} [R_{k}^{\varphi}(x_t,\ldots,x_{t+k})]$.
    \item $\mathbf{R}_{E}^{\varphi} \coloneqq \bigcup_{t \geq 1} \mathbf{R}_{t}^{\varphi}$. 
    \item $\mathbf{R}_{E} = \bigcup_{\varphi \in \mathcal{L}} \mathbf{R}_{E}^{\varphi}$. 
\end{enumerate}
For instance, we interpret $p \in \mathbf{O}_{E}$ as $p$ \emph{eventually witnesses instability}. 
\end{definition}

\begin{corollary} Let $\mu \in \mathfrak{M}_{x}^{\inv}(\mathcal{U},M)$. If $\mu$ is \emph{fim} then $\mathbb{P}_{\mu}(\mathbf{O}_{E})= \mathbb{P}_{\mu}(\mathbf{I}_{E})= \mathbb{P}_{\mu}(\mathbf{L}_{E})= 0$. 
\end{corollary}

\begin{proof}
By a similar argument to Lemma \ref{lem:01} and Proposition \ref{prop:main}, we claim that for any $\varphi \in \mathcal{L}$, $\mathbb{P}_{\mu}(\mathbf{O}_{t}^{\varphi}) = 0$. Similarly 
\begin{align*}
    \mathbb{P}_{\mu}(\mathbf{I}_{E}) \leq  \mathbb{P}_{\mu}(\mathbf{O}_{E}) = \sum_{\varphi \in \mathcal{L}} \mathbb{P}_{\mu}(\mathbf{O}_{E}^{\varphi}) = \sum_{\varphi \in \mathcal{L}} \sum_{t \geq 1} \mathbb{P}_{\mu}(\mathbf{O}_{t}^{\varphi}) = \sum_{\varphi \in \mathcal{L}} \sum_{t \geq 1} 0=0.
\end{align*}
Again replacing $\mathbf{I}_{E}$ by $\mathbf{L}_{E}$ gives the other result. 
\end{proof}

Another interpretation might involve witnessing instability \emph{via a subsequence}. We can formalize this as follows: Given a formula $\varphi(x,y) \in \mathcal{L}_{xy}(\mathcal{U})$, we say that a \emph{subsequence of $p$ witnesses instability of $\varphi$} if there exists an injection $f:\mathbb{N} \to \mathbb{N}$ such that the pushforward $f_{*}(p)$ witnesses instability of $\varphi$. We recall that $\theta(x_1,\ldots,x_n) \in f_{*}(p)$ if and only if $\theta(x_{f(1)},\ldots,x_{f(n)}) \in p$. And so, the set $\{p \in S_{\mathbf{x}}(\mathcal{U}) : \text{ a subsequence of $p$ witnesses instability for $\varphi$} \}$ can be written as 
\begin{equation*}
      \mathbf{S}^{\varphi} \coloneqq \bigcup_{\substack{f: \mathbb{N} \to \mathbb{N} \\ \text{injective}}} f_{*}^{-1}(\mathbf{O}^{\varphi}). 
\end{equation*}
We first remark that, a priori, $\mathbf{S}^{\varphi}$ may not be $\mathbb{P}_{\mu}$-measurable. It is a union of uncountably many Borel sets. 
Secondly, it is easy to find examples of measures $\mu$ such that $\mu$ is \emph{fim} and $\mathbb{P}_{\mu}(\mathbf{S}^{\varphi}) = 1$. Reconsider Example \ref{example:main} and the measure $\mu_{L}$. We recall that this measure is \emph{fim}. Moreover, we know that with probability one, the induced structure on almost all random generic types is \emph{DLO} (Example \ref{example:DLO}). Let $N = (\mathbb{Q},<)$ and $\varphi(x,y) = x <y$. Then the following two facts hold: 
\begin{enumerate}
    \item $B_{N} \subseteq \mathbf{S}^{\varphi}$.   
    \item $\mathbb{P}_{\mu_{L}}(B_{N}) = 1$. 
\end{enumerate}
This shows us that $\mathbf{S}^{\varphi}$ does not behave like $\mathbf{O}^{\varphi}$ or $\mathbf{O}^{\varphi}_{E}$. 

\subsection{Some examples} Here, we have a quick interlude on some basic examples. One might have the impression that, like in the previous section, the main ingredient in the proof of Theorem \ref{theorem:witness} is secretly commutativity. This is not the case. We consider our old friends from the Rado graph. We recall that these measures are $M$-\strong\ and thus self-commute. 

\begin{example} Recall from Example \ref{example:main} the measures $\mu_t$ for $t \in (0,1)$. We remark that 
\begin{enumerate}
    \item $\mathbb{P}_{\mu_t}(B_{N}) = 1$ where $N$ is the countable model of the Rado graph (see Example \ref{example:random}). 
    \item $B_{N} \subseteq \mathbf{I}$ and so $\mathbb{P}_{\mu_{t}}(\mathbf{I}) =\mathbb{P}_{\mu_{t}}(\mathbf{O}) = 1$. 
    \item The theory of the Rado graph is simple and therefore no formula witnesses the strict order property. In particular, this implies that for any $\varphi \in \mathcal{L}$, $\mathbf{L}^{\varphi} = \emptyset$ and thus $\mathbb{P}_{\mu}(\mathbf{L}) = 0$. 
\end{enumerate}
\end{example}

Weaker versions of tameness are not enough to derive the conclusions of Theorem \ref{theorem:witness}. Our next example is a \emph{fam} measure which behaves poorly relative to our dividing line events\footnote{The notion of \emph{fam} is a weakening of \emph{fim}. We refer the curious reader to \cite[Definition 4.1]{CGH} for an explicit definition.}

\begin{example} Let $T$ be the theory of the random triangle free graph (the Henson graph) in the language $\mathcal{L} = \{R\}$. Let $M \models T$ and $M \prec \mathcal{U}$. Consider the type $p = \{\neg R(x,b): b \in \mathcal{U}\}$. We recall from \cite{CG} that the measure $\mu \coloneqq \delta_{p}$ is \emph{fam} over any model and $\mu$ is \emph{fam} over $M$. Since $\mu$ is a type, we remark that $\supp(\mathbb{P}_{\mu}) = \{p^{(\omega)}\}$. One can prove that $\neg R(x_i,x_j) \in p^{(\omega)}$ for any $i, j \geq 1$. We claim that one can thus shatter any realization of $p^{(\omega)}$ via the edge relation. Hence $\mathbb{P}_{\mu}(\mathbf{I}) = \mathbb{P}_{\mu}(\mathbf{O}) = 1$. Notice that $\mathbb{P}_{\mu}(\mathbf{L}) =0 $ since $T$ does not have the strict order property. 
\end{example}

We now construct an example where $\mathbb{P}_{\mu}(\mathbf{O}) = \mathbb{P}_{\mu}(\mathbf{L}) = 1$ while $\mathbb{P}_{\mu}(\mathbf{I}) = 0$. 

\begin{example}\label{Ex:NIP} Consider $M= (\mathbb{R},<)$ and let $M \prec \mathcal{U}$. Let $p$ be the unique type extending $\{x > a: a \in \mathcal{U}\}$ and consider the measure $\mu \coloneqq \delta_{p}$. The measure $\mu$ is definable over $M$, but it is not \emph{fim}. Notice that $\supp(\mathbb{P}_{\mu}) = \{p^{(\omega)}\} $ and for each $n \geq 1$, $(x_{n} < x_{n+1}) \in p^{(\omega)}$. Let $\varphi(x,y) \coloneqq x <y$. Notice that $\mathbb{P}_{\mu}(\mathbf{O}^{\varphi}) = \mathbb{P}_{\mu}(\mathbf{L}^{\varphi}) = 1$ and so $\mathbb{P}_{\mu}(\mathbf{O}) = \mathbb{P}_{\mu}(\mathbf{L}) = 1$.  Since our structure is NIP, we have that $\mathbb{P}_{\mu}(\mathbf{I}) = 0$. 
\end{example}

Our model-theoretic intuition tells us the answer to the following question should be false via Shelah's theorem. 

\begin{question} Does there exist a pair of structures $(\mathcal{U},M)$ and a measure $\mu \in \mathfrak{M}_{x}^{\inv}(\mathcal{U},M)$ such that the following properties hold:  
\begin{enumerate}
    \item $\mu$ is $M$-\weak. 
    \item $\mathbb{P}_{\mu}(\mathbf{O}) = 1$ while $\mathbb{P}_{\mu}(\mathbf{I}) = \mathbb{P}_{\mu}(\mathbf{L}) = 0$?
\end{enumerate}
\end{question}

\subsection{The NIP} In the NIP setting, there are many examples of measures such that  almost all random generic types witness instability. This phenomenon is quite abundant (for a concrete example, see Example \ref{Ex:NIP}). The catch is that these measures need to not be \emph{generically stable} and in the NIP context, this means that the measure does not self-commute. Hence, the construction of the measure $\mathbb{P}_{\mu}$ in these contexts is asymmetric. This asymmetry gives us the opportunity to ask, ``what are the different expected values of $O_{n}^{\varphi}(x_{\sigma(1)},\ldots,x_{\sigma(n)})$ where $\sigma$ is an element of the symmetric group on $\{1,\ldots,n\}$? Are these values dependent on one another in some way?'' 

It turns out that if $\varphi(x,y)$ is NIP, then the average value of the measure of $O_{n}^{\varphi}(x_{\sigma(1)},\ldots,x_{\sigma(n)})$ across all permutations must tend toward 0. Why? Let's think about the extreme case. If $\mathbb{P}_{\mu}(O_{n}^{\varphi}(\bar{x})) = 1$ for every permutation of variables, then one can quickly deduce that $\varphi(x,y)$ has VC dimension at least $n$. Using this intuition,  we prove that if $\mu$ is $M$-\weak\ and $\varphi(x,y)$ is NIP then 

\begin{equation*}
     \lim_{n \to \infty} \frac{1}{n!} \sum_{\sigma \in \Sym(n)} \mathbb{P}_{\mu}(O_{n}^{\varphi}(x_{\sigma(1)}, \ldots, x_{\sigma(n)})) = 0.
\end{equation*}

One interpretation of this phenomenon is that under NIP, while one might be able to find a large measure of witnesses to instability (e.g.\ Example \ref{Ex:NIP}), those witnesses are quite fragile under permutations.

To prove the main result of this section, we use some hard bounds from the theory around \emph{VC dimension of set of permutations} to prove a soft result. It would be interesting in its own right to compute some of the hard bounds in our setting. We first recall the definition of the VC-dimension of a set of permutations form \cite{cibulka2012tight}. 

\begin{definition} Let $\Sym(n)$ denote the permutation group on $\{1,\ldots,n\}$ and $A \subseteq \Sym(n)$. Then the \emph{VC-dimension of the set $A$} is defined to be the largest integer $k$ such that the set of restrictions of the permutations in $A$ to some $k$-tuple of positions is the set of all $k!$ permutation patterns. We let $r_{k}(n)$ be the maximum size of a set of $n$-permutations with VC-dimension $k$. 
\end{definition}

There is a clear connection between the VC dimension of a set of permutations and our setting. 

\begin{lemma}\label{lemma:VC} Fix a formula $\varphi(x,y)$ and a type $p \in S_{x_1,\ldots,x_m}(M)$. If 
\begin{equation*}
    A \coloneqq \{\sigma \in \Sym(m): O^{\varphi}_n(x_{\sigma(1)},\ldots,x_{\sigma(m)}) \in p\} > r_{d}(m),
\end{equation*}
Then the VC dimension of $\{\varphi(x,b): b \in \mathcal{U}^{y}\}$ is strictly greater than $d$. 
\end{lemma}

\begin{proof} Since our permutation set is strictly larger than $r_{d}(m)$, we know that the VC dimension of the permutation set is strictly greater than $d$. Hence there exists $j_1 < \ldots < j_d$ in $\{1,\ldots,m\}$ such that $\{\sigma|_{j_1,\ldots,j_d}: \sigma \in A\} = \Sym(\{j_1,\ldots,j_d\})$. Now let $(a_1,\ldots,a_m) \models p$. We claim the set $\{a_{j_l}: 1 \leq l \leq d\}$ is shattered by $\{\varphi(x,b): b \in \mathcal{U}\}$. 
Indeed, if $K \subseteq \{j_1,\ldots,j_d\}$, let $\tau$ be an element of $\Sym(\{j_1,\ldots,j_d\})$ which places the elements of $K$ before the elements 
of $\{j_1,\ldots,j_d\} \backslash K$. Then there exists some $\tau' \in A$ such that $\tau'|_{j_1,\ldots,j_{d}} = \tau$. Thus, $\models O_{m}^{\varphi}(a_{\tau'(1)},\ldots,a_{\tau'(m)})$ implies
\begin{align*}
    \mathcal{U} &\models \exists y \left( \bigwedge_{i = 1}^{|K|} \varphi(a_{\tau'(i)},y) \wedge \bigwedge_{i = |K| + 1}^{m} \neg \varphi(a_{\tau'(i)},y) \right) \\ &\Longrightarrow
    \mathcal{U} \models \exists y \left( \bigwedge_{i \in K}  \varphi(a_{i},y) \wedge \bigwedge_{i \in \{j_1,\ldots,j_d\} \backslash K} \neg \varphi(a_{i},y) \right) 
\end{align*}
And so, for some $b \in \mathcal{U}$, we have that $\{a_{j_1},\ldots,a_{j_d}\} \cap \{\varphi(\mathcal{U},b)\} = \{a_i: i \in K\}$. 
\end{proof}

The next theorem is by Cibulka and Kyn\v{c}l \cite[Theorem 1.1]{cibulka2012tight}. 

\begin{fact}\label{fact:bound} For every $t \geq 1$, we have that $r_{2t + 2}(m) = 2^{\Theta(m\alpha(m)^{t})}$ and $r_{2t+3}(m) = 2^{\Theta(m\alpha(m)^{t} \log \alpha(m))}$ where $\alpha(m)$ is the inverse Ackermann function. 
\end{fact}

A straightforward consequence of the above fact is the following.

\begin{fact}\label{fact:exist} For every pair of integers $(n,d)$ there exists an integer $m$ such that $m!/n > r_{d}(m)$. 
\end{fact}

\begin{proof} This follows directly from Fact \ref{fact:bound}. Intuitively, this holds because $m!$ grows much faster than ``$2$ to the inverse Ackermann of $m$ to the $d$" for fixed $d$. Dividing by $n$ is more or less garnish. 
\end{proof}

Finally we come to the main theorem of this subsection. 

\begin{theorem}\label{theorem:average} Let $\mu \in \mathfrak{M}_{x}^{\inv}(\mathcal{U},M)$ and suppose that $\mu$ is $M$-\weak. Let $\varphi(x,y)$ be an NIP formula. Then 
\begin{equation*}
        \lim_{n \to \infty} \frac{1}{n!} \sum_{\sigma \in \Sym(n)} \mathbb{P}_{\mu}(O_{n}^{\varphi}(x_{\sigma(1)}, \ldots, x_{\sigma(n)})) = 0. 
\end{equation*}
\end{theorem}

\begin{proof} Since $\varphi(x,y)$ is NIP, there exists a $d$ such that the VC-dimension of the set system $\{\varphi(x,b): b \in \mathcal{U}\}$ is less than or equal to $d$. We notice that if $n < n'$, then
\begin{equation*}
    \frac{1}{n!} \sum_{\sigma \in \Sym(n)} \mathbb{P}_{\mu}(O_{n}^{\varphi}(x_{\sigma(1)}, \ldots,x_{\sigma(n)})) \geq \frac{1}{n'!} \sum_{\sigma \in \Sym(n')} \mathbb{P}_{\mu}(O_{n'}^{\varphi}(x_{\sigma(1)}, \ldots x_{\sigma(n')})). 
\end{equation*}
Therefore, our terms are decreasing and bounded below, and the sequence converges to some number $c$. We argue that $c = 0$. If not, then $c > 0$. Choose $n_0 \in \mathbb{N}$ such that $0 < \frac{1}{n_0} < c$. By Fact \ref{fact:exist}, there exists a number $m$ such that $m!/n_0 > r_{d}(m)$. Now we have that 
\begin{align*}
    \frac{1}{n_0 } < c 
 &<\frac{1}{m!} \sum_{\sigma \in \Sym(m)} \mu^{(m)}(O_{m}^{\varphi}(x_{\sigma(1)},\ldots,x_{\sigma(m)})) \\ &= \int_{S_{\bar{x}}(M)} \frac{1}{m!} \sum_{\sigma \in \Sym(m)} \mathbf{1}_{O_{m}^{\varphi}(x_{\sigma(1)},\ldots,x_{\sigma(m)})} d\mu^{(m)} \\
 &\Longrightarrow \frac{m!}{n_0} <  \int_{S_{\bar{x}}(M)} \sum_{\sigma \in \Sym(m)}\mathbf{1}_{O_{m}^{\varphi}(x_{\sigma(1)},\ldots,x_{\sigma(m)})} d\mu^{(m)}\\
 &\Longrightarrow \exists p_* \in \supp(\mu^{(m)}) \text{ such that } \sum_{\sigma \in \Sym(m)}\mathbf{1}_{O_{m}^{\varphi}(x_{\sigma(1)},\ldots,x_{\sigma(m)})}(p_*)  > \frac{m!}{n_0} \\
 &\Longrightarrow |\{\sigma \in \Sym(m): O_{m}^{\varphi}(x_{\sigma(1)},\ldots,x_{\sigma(m)}) \in p_{*}\}| > \frac{m!}{n}.
\end{align*}
And so by our choice of $m$, 
\begin{equation*}
    |\{\sigma \in \Sym(m): O_{m}^{\varphi}(x_{\sigma(1)},\ldots,x_{\sigma(m)}) \in p_{*}\}| > r_{d}(m).
\end{equation*}
By Lemma \ref{lemma:VC}, we have that the VC dimension of $\{\varphi(x,b):b \in \mathcal{U}\}$ is strictly greater than $d$, which is a contradiction. 
\end{proof}

\begin{remark} The previous result actually gives us another proof of Theorem \ref{theorem:witness} for generically stable measures in the NIP setting. Suppose $T$ is NIP, $\mu \in \mathfrak{M}_{x}^{\inv}(\mathcal{U},M)$ is generically stable over $M$. Then $\mathbb{P}_{\mu}(\mathbf{O}) = 0$. Notice that for any $\mathcal{L}$-formula $\varphi(x,y)$, 
\begin{align*}
    \mathbb{P}_{\mu}(\mathbf{O}^{\varphi}) &= \lim_{n \to \infty} \mu^{(n)}(O_{n}^{\varphi}(x_1,\ldots,x_n)) \\
    &=\lim_{n \to \infty} \frac{1}{n!}\sum_{\sigma \in \Sym(n)} \mu^{(n)}(O_{n}^{\varphi}(x_1,\ldots,x_n)) \\ 
    &\overset{(*)}{=}\lim_{n \to \infty} \frac{1}{n!}\sum_{\sigma \in \Sym(n)} \mu^{(n)}(O_{n}^{\varphi}(x_{\sigma(1)},\ldots,x_{\sigma(n)})) \\ 
    &\overset{(**)}{=}0
\end{align*}
Equation $(*)$ follows from the fact that $\mu$ self-commutes. Equation $(**)$ is Theorem \ref{theorem:average}. Hence $\mathbb{P}_{\mu}(\mathbf{O}) = 0$ since $\mathbf{O}$ is the countable union of sets of measures $0$. 
\end{remark}

We consider an example where the limit is equal to 1.

\begin{remark} Notice if we consider $\mu_t$ from Example \ref{example:main} and set $\varphi(x,y) \coloneqq R(x,y)$, the edge relation, we see that 
\begin{align*}
        \lim_{n \to \infty} \frac{1}{n!} \sum_{\sigma \in \Sym(n)} \mathbb{P}_{\mu_{t}}(O_{n}^{\varphi}(x_{\sigma(1)}, \ldots, x_{\sigma(n)})) = \lim_{n \to \infty} \frac{1}{n!} \sum_{\sigma \in \Sym(n)} 1 = 1,
\end{align*}
because $\mu$ is $M$-\strong\ (i.e.\ self-commutes) and $\mathbb{P}_{\mu_{t}}(O_{n}^{R}(x_1,\ldots,x_n)) = 1$. 
\end{remark}

\bibliographystyle{plain}
\bibliography{refs}

\end{document}